\documentclass[12pt,reqno]{amsart}
 \usepackage[usenames,dvipsnames]{color}
\usepackage{graphicx}
\usepackage{fullpage}
\usepackage{amsthm}
\usepackage{shuffle}
\usepackage{float}
\usepackage[colorlinks=true,%linkcolor=RoyalBlue,
   citecolor=cyan,urlcolor=blue]{hyperref} \def\dessin#1#2{\includegraphics[#1]{#2}}

\usepackage{wrapfig}
\usepackage{amssymb}
\usepackage{amsmath}

\def\bleu{\textcolor{blue}}

\makeatletter
\newtheorem*{rep@theorem}{\rep@title}
\newcommand{\newreptheorem}[2]{
\newenvironment{rep#1}[1]{
\def\rep@title{#2 \ref{##1}}
\begin{rep@theorem}}
{\end{rep@theorem}}}
\makeatother
\def\auteur#1{{\sc #1}}
\def\titreref#1{{\em #1}}
\def\vol#1{{\bf #1}}
\newtheorem{thm}{Theorem}

\newtheorem{conj}{Conjecture}
\newtheorem{rmk}{Remark}
\newtheorem{prop}{Proposition}

\newtheorem{defn}{Definition}
\newtheorem{alg}{Algorithm}
\newreptheorem{thm}{Theorem}

\numberwithin{equation}{section}
\numberwithin{thm}{section}
\numberwithin{lemma}{section}
\numberwithin{conj}{section}
\numberwithin{rmk}{section}
\numberwithin{prop}{section}
\numberwithin{cor}{section}
\numberwithin{defn}{section}
\numberwithin{alg}{section}

\newcommand{\area}{\operatorname{area}}

\newcommand{\dinv}{\operatorname{dinv}}
\newcommand{\arm}{\operatorname{arm}}
\newcommand{\leg}{\operatorname{leg}}
\newcommand{\coarm}{\operatorname{coarm}}
\newcommand{\coleg}{\operatorname{coleg}}

\newcommand{\tdinv}{\operatorname{tdinv}}
\newcommand{\mdinv}{\operatorname{maxtdinv}}
\newcommand{\rank}{\operatorname{rank}}
\newcommand{\ides}{\operatorname{ides}}
\newcommand{\des}{\operatorname{des}}
\newcommand{\pides}{\operatorname{pides}}
\newcommand{\p}{\operatorname{comp}}
\newcommand{\ret}{\operatorname{ret}}
\def\OC#1{OC[#1]}
\def\IC#1{IC[#1]}
\newcommand{\shape}{\operatorname{sh}}

\def\define#1{\bleu{\emph{#1}}}

\def\Parking{\mathrm{Park}}
\def\U{\mathbf{u}}
\def\V{\mathbf{v}}
\def\tabl{T}
\def\path{\gamma}
\def\X{\mathbf{x}}
\def\SL{\mathrm{SL}}
\def\park{\pi}
\def\tfrac#1#2{{\textstyle\frac{#1}{#2}}}

\def\un{\cdot\mathbf{1}}
\def\unn{\cdot (-\mathbf{1})^n}

\def\LLT{\mathrm{LLT}}
\def\N{\mathbb{N}}

\def\A{\Gamma}

\def\correction{\displaywidth=\parshapelength\numexpr\prevgraf+2\relax}
\def\pref#1{{\rm (\ref{#1})}}

\def\Hik{H}

\DeclareMathOperator{\Qop}{\mathbf{Q}}
\DeclareMathOperator{\spl}{Split}
\DeclareMathOperator{\BC}{\mathbf{C}}
\DeclareMathOperator{\BB}{\mathbf{B}}
\DeclareMathOperator{\D}{\mathbf{D}}
\DeclareMathOperator{\Be}{\mathbf{e}}
\DeclareMathOperator{\BE}{\mathbf{E}}
\DeclareMathOperator{\BF}{\mathbf{F}}
\DeclareMathOperator{\Bf}{\mathbf{f}}
\def\ue{\underline{e}}
\def\uh{\underline{h}}

\def\charac{\raise 2pt\hbox{\large$\chi$}}

\title{Compositional \lowercase{$(km,kn)$}--Shuffle Conjectures}
\author{F.~Bergeron}
\address{\href{http://bergeron.math.uqam.ca}{D\'epartement de Math\'ematiques, Lacim, UQAM.}}
  \email{\href{mailto:bergeron.francois@uqam.ca}{bergeron.francois@uqam.ca}}
 \author{A.~Garsia}
\address{\href{http://www.math.ucsd.edu/~garsia/}{Department of Mathematics, UCSD.}}
 \email{\href{mailto:garsia@math.ucsd.edu}{garsia@math.ucsd.edu}}
 \author{E.~Leven}
\address{Department of Mathematics, UCSD.}
 \email{\href{mailto:esergel@ucsd.edu}{esergel@ucsd.edu}}
 \author{G.~Xin}
 \address{\href{http://cfc.nankai.edu.cn/homepage/xin/}{School of mathematical science, Capital Normal University, PR China}}
 \email{\href{mailto:guoce.xin@gmail.com}{guoce.xin@gmail.com}}
 \date{July 6, 2014. This work was supported by NSERC, NSF grant DGE 1144086 and NSFC(11171231).}

\begin{document}

%%%%%%%%%%%%%%%%%%%%%%%%%%%%%%%%%%

\begin{abstract}
In 2008, Haglund, Morse and Zabrocki \cite{classComp} formulated a Compositional form of the Shuffle Conjecture of Haglund {\em et al.} \cite{Shuffle}. In very recent work, Gorsky and Negut  by combining their discoveries \cite{GorskyNegut}, \cite{NegutFlags} and \cite{NegutShuffle}, with the work of Schiffmann-Vasserot \cite{SchiffVassMac} and \cite{SchiffVassK} on the symmetric function side and the work of Hikita \cite{Hikita} and Gorsky-Mazin \cite{GorskyMazin} on the combinatorial side, were led to formulate an infinite family of conjectures that extend the original Shuffle Conjecture of \cite{Shuffle}. In fact, they formulated one conjecture for each pair $(m,n)$ of coprime integers. This work of Gorsky-Negut leads naturally to the question as to where the Compositional Shuffle Conjecture of Haglund-Morse-Zabrocki fits into these recent developments. Our discovery here is that there  is a compositional extension of the Gorsky-Negut Shuffle Conjecture for each pair $(km,kn)$, with $(m,n)$ co-prime and $k > 1$.
\end{abstract}

\maketitle
 \parskip=0pt
{ \setcounter{tocdepth}{1}\parskip=0pt\footnotesize \tableofcontents}
\parskip=8pt  
\parindent=20pt

%%%%%%%%%%%%%%%%%%%%%%%%%%%%%%%%%%

\section*{Introduction} 
The subject of the present investigation  has its origin, circa 1990, in a  effort  to obtain a representation theoretical setting for the Macdonald $q,t$-Kotska coefficients. This effort culminated in  Haimain's proof, circa  2000, of the $n!$ conjecture (see \cite{natBigraded})
by means of the Algebraic Geometry of the Hilbert Scheme. 
In the 1990's, a concerted effort by many researchers led to a variety of conjectures  tying the theory of Macdonald Polynomials to 
the representation theory of Diagonal Harmonics and the combinatorics
of parking functions. More  recently, this subject has been literally flooded 
with connections with other areas of mathematics such as: the Elliptic Hall Algebra of Shiffmann-Vasserot, the Algebraic Geometry of Springer Fibers of Hikita, the Double Affine Hecke Algebras of Cherednik,  the HOMFLY polynomials,  and the truly fascinating Shuffle Algebra of symmetric functions. This has brought to the fore a variety of symmetric function operators with close connection to the extended notion of \define{rational parking functions}. The present work results from an ongoing effort to express and deal with these new developments  in a
language that is more accessible to the algebraic 
combinatorial audience.
This area of investigation  involves many aspects of symmetric function theory, including a central role played by Macdonald polynomials, as well as some of their closely related  symmetric function operators.  
One of the alluring characteristic features of   these operators  is that they appear to control in a rather  surprising manner  combinatorial properties of 
 rational  parking functions. A close investigation of these connections led us to a variety of new discoveries and conjectures in this area which in turn  
 should  open up a variety of open problems in 
Algebraic Combinatorics as well as in the above 
 mentioned areas.

\section{The previous shuffle conjectures}
We begin by reviewing the statement of the  Shuffle Conjecture of Haglund {\em et al.}  (see \cite{Shuffle}). In Figure \ref{fig:Park2} we have an example of two convenient ways to represent a parking function: a two-line array and a tableau. 
\begin{figure}[ht]
\begin{center}
\begin{displaymath}
\begin{bmatrix}
4 &  6 &  8 &  1 &  3  & 2 &  7 &  5 \\
0 &  1 &  2 &  2 &  3  & 0 &  1 &  1 \\ 
\end{bmatrix} \qquad\qquad \Longleftrightarrow \qquad\qquad
\vcenter{\hbox{\dessin{width=130pt}{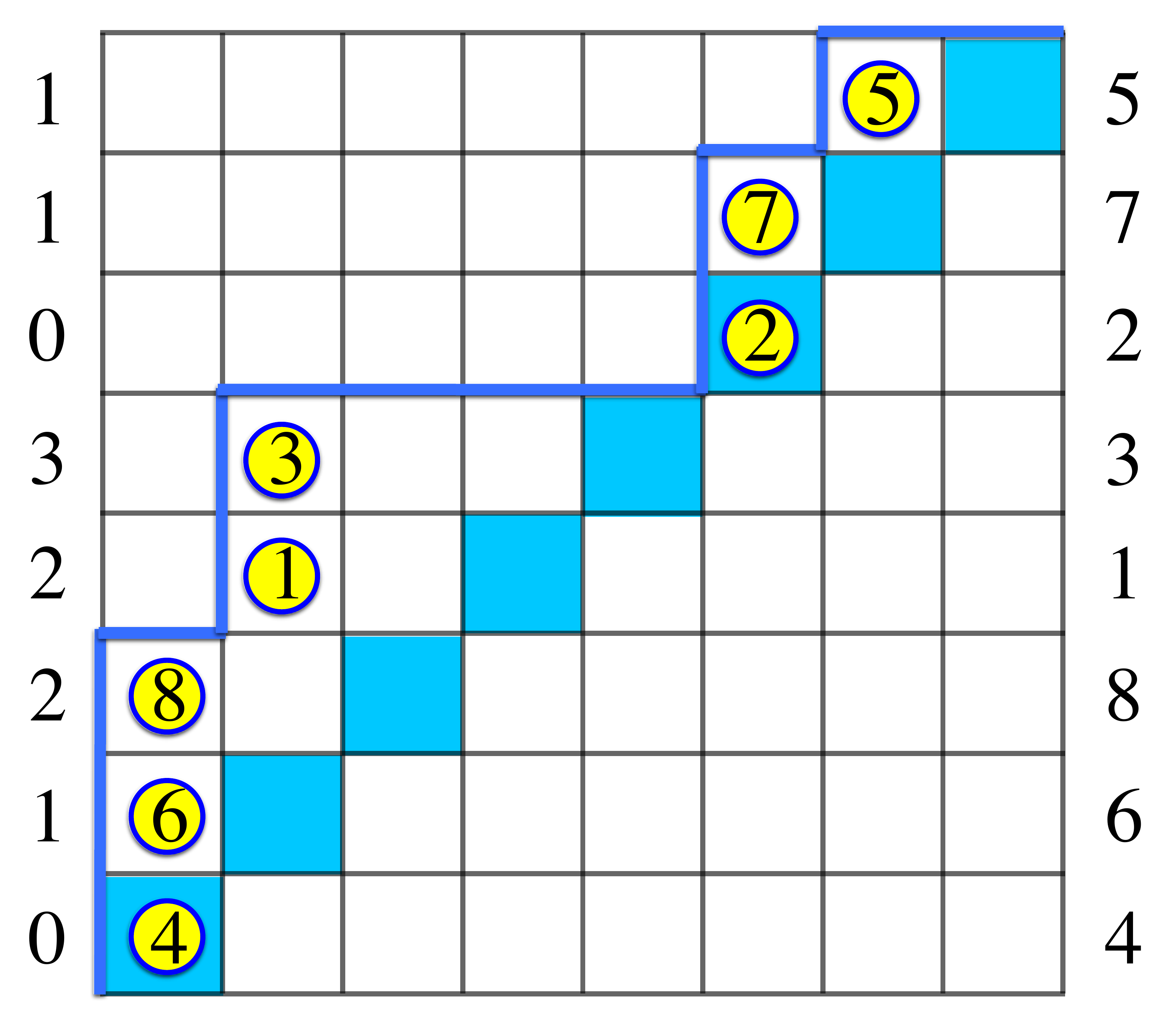}}}
\end{displaymath}
\caption{Two representations of a parking function}
\label{fig:Park2}
\end{center}
\end{figure}
The tableau on the right is constructed by first choosing a \hbox{Dyck  path.}  Recall that this is a path   in the $n\times n$ lattice square that goes from $(0,0)$ to $(n,n)$ by \define{north} and \define{east} steps, always remaining weakly above the main diagonal (the shaded cells). The lattice cells adjacent and to the east of north steps are filled with {\define{cars} $ 1,2,\dots,n $}  in a column-increasing  manner. The numbers on the top of the two-line array are the cars as we read them by rows, from bottom to top. The numbers on the bottom of the two line array are the \define{area} numbers, which are obtained by successively counting the number of lattice cells between a \define{north} step and the main diagonal.
All the necessary statistics of a parking function $\park$ can be immediately obtained from the corresponding two line array 
$$
\park:= \Big[\begin{matrix}
v_1 & v_2 & \cdots  &  v_n\\[-4pt] 
u_1 & u_2 & \cdots  &  u_n\\ 
\end{matrix}\Big].$$
To begin we let
\begin{eqnarray*}
\area(\park)&:=&\sum_{i=1}^n u_i,
\qquad {\rm and}\\
\dinv(\park)&:=& 
\sum_{1 \leq i<j \leq n } \charac\big(\,( u_i=u_j\  \&\  v_i<v_j)\quad {\rm or}\quad
 (u_i=u_j+1\  \&\  v_i>v_j)\,\big),
\end{eqnarray*}
where $\charac(-)$ denotes the function that takes value $1$ if its argument is true, and $0$ otherwise.
Next we define $\sigma(\park)$ to be the permutation obtained by successive right to left readings of the components of the vector $(v_1 , v_2 , \dots  ,  v_n)$ according to decreasing values of $u_1 , u_2 , \dots,  u_n$.  Alternatively, $\sigma(\park)$ is also obtained by reading the cars, in the tableau, from right to left  by diagonals and from the highest diagonal to the lowest. Finally, we denote by $\ides(\park)$  the descent set of the inverse of $\sigma(\park)$. 
This given, in \cite{Shuffle} Haglund {\em et al.} stated the following.  
\begin{conj}[HHLRU-2005] \label{conjHHLRU}
For all $n \geq 1$,
\begin{equation} \label{HHLRU}
 \nabla e_n(\X) \,=\, \sum_{\park \in\Parking_n } t^{\area(\park) }q^{\dinv(\park)} F_{\ides(\park)}[\X]
\end{equation}
\end{conj}
Here, $\Parking_n$ stands for the set of all parking functions in the $n\times n$ lattice square. Moreover, for a subset $S \subseteq \{1,2,\cdots, ,n-1\}$, we denote by $F_S[\X]$  the corresponding Gessel fundamental quasi-symmetric function homogeneous of degree $n$ (see \cite{Gessel}). Finally, $\nabla$ is the symmetric function\footnote{See Section~\ref{secsymm} for a quick review of the usual tools for calculations with symmetric functions, and operators on them.} operator introduced in \cite{SciFi}, with eigenfunctions  the modified Macdonald polynomial basis $\{\widetilde{H}_\mu[\X;q,t] \}_{\mu}$, indexed by partitions $\mu$.

Let us now recall that a result of Gessel implies
 that
a  homogeneous symmetric function $f[\X]$ of degree $n$ has an expansion of the form
\begin{displaymath}
f[\X] \,=\, \sum_{\sigma \in S_n} c_\sigma F_{\ides(\sigma)}[\X],
\qquad\qquad ( \, \ides(\sigma) = \des(\sigma^{-1}) \, )
\end{displaymath}
if and only, if for all partitions $\mu =(\mu_1,\mu_2,\dots \mu_k)$ of $n$, we have 
\begin{displaymath}
\left\langle f \,,\, h_\mu \right\rangle
\,=\, \sum_{\sigma \in S_n} c_\sigma \, \charac \big(\sigma \in E_1\shuffle E_2 \shuffle \cdots \shuffle E_k \big),
\end{displaymath}
where $\langle- \,,\,- \rangle$ denotes the Hall scalar product of symmetric functions,  $E_1, E_2, \dots, E_k$ are successive segments of the word $123\cdots n$ of respective lengths $\mu_1,\mu_2,\dots ,\mu_k$,
and the symbol $E_1\shuffle E_2 \shuffle \cdots \shuffle E_k$ denotes the collection of permutations obtained by shuffling in all possible ways the words $E_1,E_2,\dots ,E_k$.
Thus \pref{HHLRU} may be restated as 
\begin{equation} \label{I.2}
\left\langle \nabla e_n \,,\, h_\mu \right\rangle
\, = \sum_{\park \in\Parking_n} t^{\area (\park)} q^{\dinv(\park)} \charac \big( \sigma(\park) \in E_1\shuffle E_2 \shuffle \cdots \shuffle E_k \big)
\end{equation}
for all  $ \mu \vdash n$, which is the original form of the Shuffle Conjecture. Recall that it is customary  to write $\mu\vdash n$, when $\mu$ is a partition of $n$.

For about 5 years from its formulation, the Shuffle Conjecture appeared untouchable for lack of any recursion satisfied by both sides of the equality.  However, in the fall of 2008, Haglund, Morse, and Zabrocki \cite{classComp} made a discovery that is nothing short of spectacular. They discovered that two slight deformations $\BC_a$ and $\BB_a$ of the well-known Hall-Littlewood operators, combined with $\nabla$,  yield considerably finer versions of the Shuffle conjecture. For any $a\in\N$, they define the operators $\BC_a$ and $\BB_a$, acting on  symmetric polynomials  $f[\X]$, as follows.
\begin{eqnarray} 
   \BC_a f[\X] &=& (-q)^{1-a} f\!\left[\X-{(q-1)}/{(qz)}\right]\, \sum_{m \geq 0} z^m h_m[\X]\, \Big|_{z^a}, \qquad {\rm and} \label{defC}\\ 
   \BB_b  f[\X] &=& f\!\left[ \X + \epsilon\,(1-q)/z\right]\, \sum_{m \geq 0} z^m e_m[\X]\, \Big|_{z^b},
\end{eqnarray}
where $(-)\big|_{z^a}$ means that we take the coefficient of $z^a$ in the series considered. We use here ``plethystic'' notation   which is described in more details in Section~\ref{secsymm}.
Haglund, Morse, and Zabrocki also introduce a  new statistic on paths (or parking functions), the  \define{return composition} 
    $$\p(\park)=(a_1,a_2,\dots, a_\ell),$$ 
 whose parts are the sizes of the intervals between successive diagonal hits of the Dyck path of $\park$, reading from left to right. As usual we write $\alpha\models n$, when $\alpha$ is a composition of $n$, {\em i.e.} $n=a_1+\ldots+a_\ell$ with the $a_i$ positive integers, and set $\BC_\alpha$ for the product $\BC_{a_1} \BC_{a_2} \cdots \BC_{a_\ell}$, with a similar convention for $\BB_\alpha$.
This given, their discoveries led them to state the following two conjectures\footnote{To make it clear that we are applying an operator to a constant function such as ``$\mathbf{1}$'', we add a dot between this operator and its argument.}.
\begin{conj} [HMZ-2008] \label{conjHMZC}
For any composition $\alpha$ of $n$,
$$
  \nabla \BC_\alpha \un \,=\sum_{\p(\park)=\alpha} t^{\area(\park)}q^{\dinv(\park)}F_{\ides(\park)}[\X],
$$
 where the sum is over parking functions in the $n\times n$ lattice square with composition \hbox{equal to $\alpha$.}
\end{conj}

\begin{conj}  [HMZ-2008]\label{conjHMZB}
For any composition $\alpha$ of $n$,
$$
  \nabla \BB_\alpha \un \,=\sum_{\p(\park)\preceq \alpha}t^{\area(\park)}q^{\dinv_\alpha(\park)}F_{\ides(\park)}[\X],
$$
where ``$\dinv_\alpha$'' is a suitably $\alpha$-modified $\dinv$ statistic, and the sum is over parking functions in the $n\times n$ lattice square  with composition finer than $\alpha$.
\end{conj}

We first discuss Conjecture \ref{conjHMZC}, referred to as the Compositional Shuffle Conjecture, and we will later come back to Conjecture 1.3. Yet another use of Gessel's Theorem shows that Conjecture \ref{conjHMZC} is equivalent to the family of identities
\begin{equation} \label{I.6}
\left\langle \nabla \BC_\alpha \un, h_\mu\right\rangle = 
\sum_{\p(\park)=\alpha} 
t^{\area(\park)}q^{\dinv(\park)} \charac \big( \sigma(\park) \in E_1\shuffle E_2 \shuffle \cdots \shuffle E_k \big),
\end{equation}
where, as before, the parts of $\mu$ correspond to the cardinalities of the $E_i$.
The fact that Conjecture \ref{conjHMZC} refines the Shuffle Conjecture is due to the identity
\begin{equation} \label{I.7}
\sum_{\alpha \models n} \BC_\alpha \un \,=\, e_n,
\end{equation}
hence summing \pref{I.6} over all compositions $\alpha\models n$ we obtain \pref{I.2}. 

Our main contribution here is to show that a suitable extension of the Gorsky-Negut Conjectures (NG Conjectures) to the non-coprime case leads to the formulation of an infinite variety of new Compositional Shuffle conjectures, widely extending both the NG and the HMZ Conjectures. To state them we need to briefly review the Gorsky-Negut Conjectures in a manner that most closely resembles the classical Shuffle conjecture.

%%%%%%%%%%%%%%%%%%%%%%%%%%%%%%%%%%

\section{The coprime case} \label{sec:coprime}

Our main actors on the symmetric function side are the operators $\D_k$ and $\D_k^*$, introduced in \cite{qtCatPos}, whose action on a symmetric function $f[\X]$ are defined by setting respectively
\begin{eqnarray}
   \D_k f[\X] &:=& f\!\left[\X+{M}/{z}\right] \sum_{i \geq 0}(-z)^i e_i[\X] \Big|_{z^k},\qquad {\rm and}  \label{defD}\\ 
   \D_k^*f[\X]&:=& f\!\left[\X-{\widetilde{M}}/{z}\right] \sum_{i \geq 0}z^i h_i[\X] \Big|_{z^k}.\label{defDetoile}
\end{eqnarray}
with $M:=(1-t)(1-q)$ and $\widetilde{M}:=(1-1/t)(1-1/q)$.
The focus of the present work is the algebra of symmetric function operators generated by the family $\{ \D_k \}_{k \geq 0}$. Its connection to the algebraic geometrical developments is that this algebra is a concrete realization of a portion of the Elliptic Hall Algebra studied Schiffmann and Vasserot in \cite{elliptic2}, \cite{SchiffVassMac}, and \cite{SchiffVassK}. Our conjectures are expressed in terms of a family of operators $\Qop_{a,b}$ indexed by pairs of positive integers $a,b$. Here, and in the following, we  use  the notation $\Qop_{km,kn}$, with $(m,n)$ a coprime pair of non-negative integers and $k$ an arbitrary positive integer. In other words, $k$ is the greatest common divisor of $a$ and $b$, and $(a,b)=(km,kn)$.

\begin{wrapfigure}[8]{R}{2.8cm} \centering
    \vskip-5pt
     \dessin{height=1.5 in}{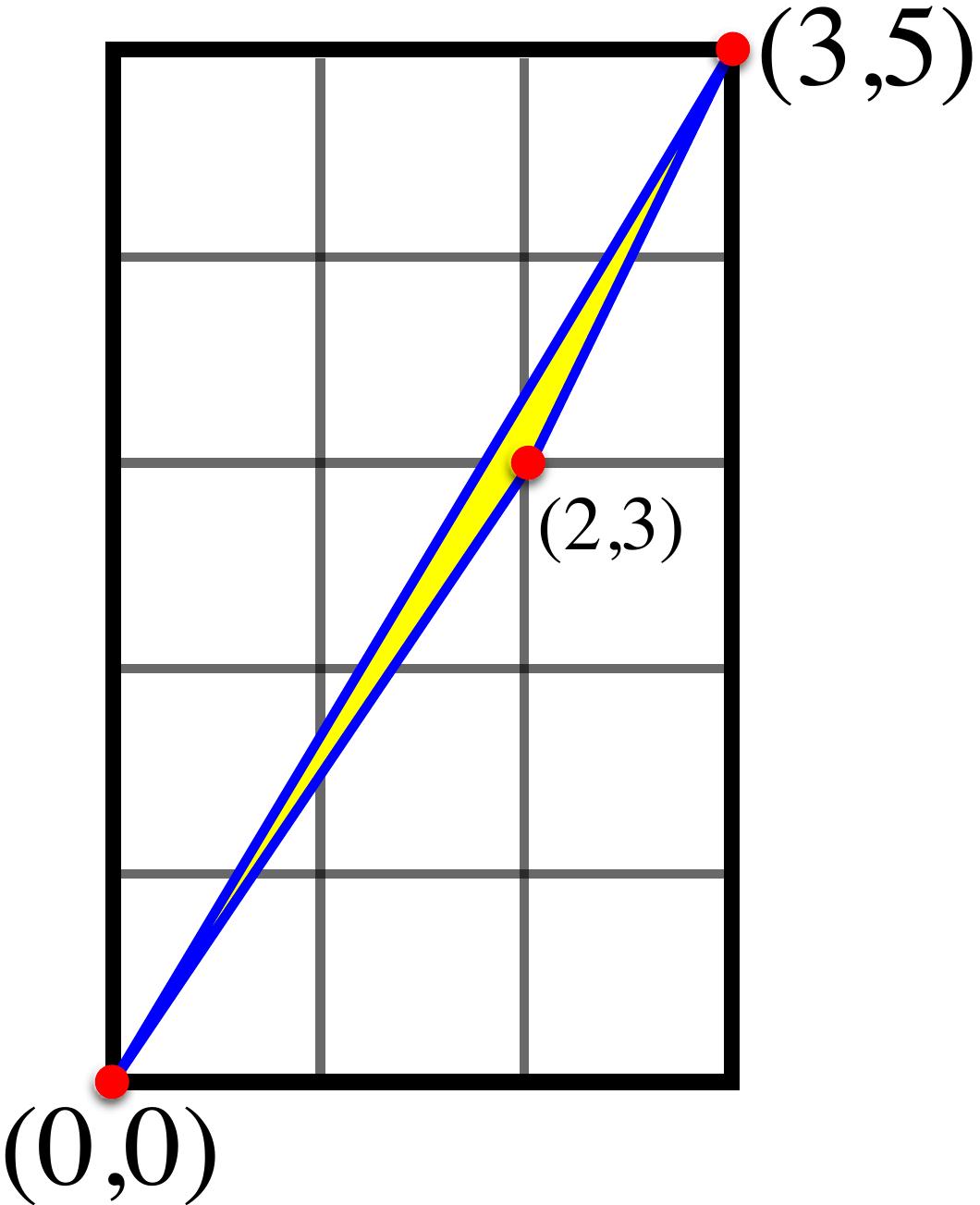}
\end{wrapfigure}
Restricted to the coprime case, the definition of the operators $\Qop_{m,n}$ is first illustrated in a special case. For instance, to obtain $\Qop_{3,5}$ we start by drawing the $3\times 5$ lattice with its diagonal (the line $(0,0) \to (3,5)$) as depictedf  in the adjacent figure. Then we look for the lattice point $(a,b)$ that is closest to and below the diagonal. In this case $(a,b)=(2,3)$. This yields the decomposition $(3,5)=(2,3)+(1,2)$, and unfolding the recursivity we get
\begin{eqnarray} 
 \Qop_{3,5} &=& \frac{1}{M} \left[ \Qop_{1,2} , \Qop_{2,3} \right]\nonumber\\
&=& \frac{1}{M} \left(
 \Qop_{1,2} \Qop_{2,3} - \Qop_{2,3} \Qop_{1,2} \right).\correction \label{1.2}
\end{eqnarray}

%\begin{figure}[ht]
%\begin{center}
%\includegraphics[width=1in]{dia35.pdf}
%\hskip 30pt
%\includegraphics[width=.8in]{dia23.pdf}
%\caption{The $3 \times 5$, and $2 \times 3$ lattices.}
%\label{fig:dias35}
%\end{center}
%\end{figure}

\begin{wrapfigure}[5]{R}{2.8cm} \centering
    \vskip-15pt
     \dessin{height=1 in}{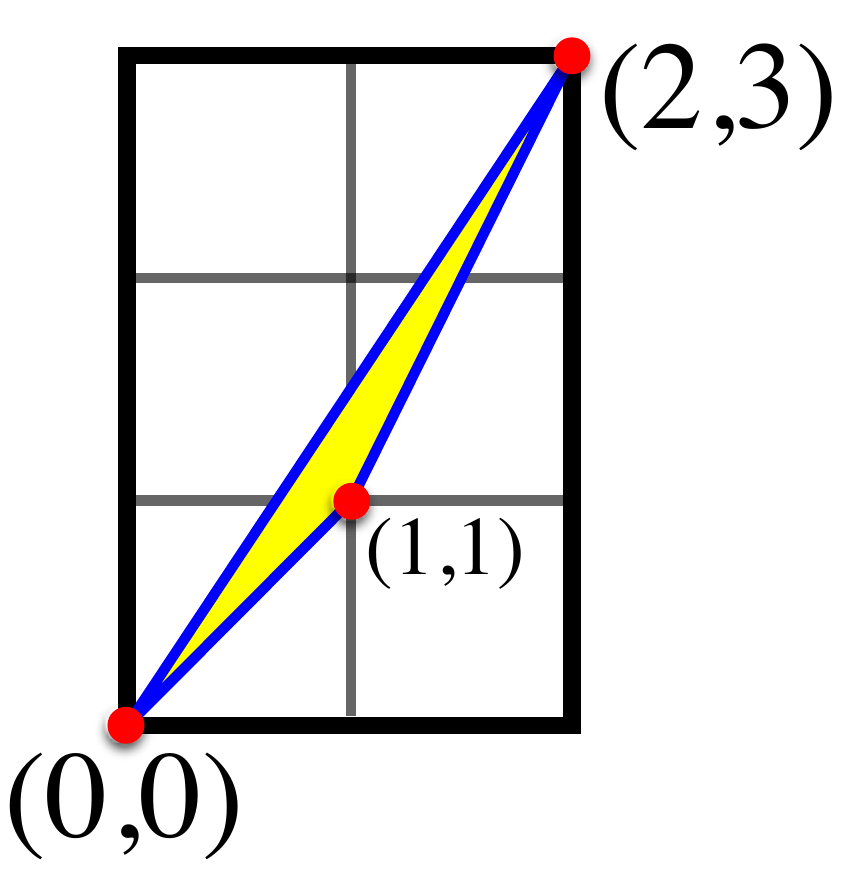}
\end{wrapfigure}
We must next work precisely in the same way with the $2\times 3$ rectangle, as indicated in the adjacent figure. We obtain the decomposition $(2,3)=(1,1)+(1,2)$ and recursively set 
\begin{equation} \correction \label{1.3}
\Qop_{2,3} = \frac{1}{M} \left[ \Qop_{1,2} , \Qop_{1,1} \right]
= \frac{1}{M} \left(
\Qop_{1,2} \Qop_{1,1} - \Qop_{1,1} \Qop_{1,2} \right).
\end{equation}
Now, in this case, we are done, since it turns out that  we can set
\begin{equation} \label{1.4}
\Qop_{1,k} = \D_k.
\end{equation}
In particular by combining \pref{1.2}, \pref{1.3} and \pref{1.4} we obtain
\begin{equation} \label{1.5}
\Qop_{3,5}
= \frac{1}{M^2}
\left( \D_2 \D_2 \D_1 - 2 \D_2 \D_1 \D_2 + \D_1 \D_2 \D_2 \right).
\end{equation}
%In the general case, a precise definition exploits an elementary number theoretical lemma that characterizes the closest lattice point $(a,b)$ below  the line $(0,0) \to (m,n)$. We then let $(c,d)=(m,n)-(a,b)$ and write
%\begin{displaymath} \spl(m,n) = (a,b)+(c,d). \end{displaymath}
%Observe that, when the pair $(m,n)$ is coprime, then $(a,b)$ is also necessarily so.
%We can thus recursively define
%\begin{equation} \label{1.6}
%\Qop_{m,n} = 
%\begin{cases}
%\tfrac{1}{M} \left[ \Qop_{c,d}, \Qop_{a,b} \right] & \hbox{if } m>1 \hbox{ and } \spl(m,n)=(a,b)+(c,d), \\[6pt]
%\, \D_n & \hbox{if } m=1.
%\end{cases}
%\end{equation}
To give a precise general definition of the $\Qop$ operators  we use the following elementary number theoretical characterization of the closest lattice point $(a,b)$ below the line $(0,0) \to (m,n)$. We observe that by construction $(a,b)$ is coprime. See \cite{newPleth} for a proof.
\begin{prop} \label{prop3.3}
For any  pair of coprime integers $m,n> 1$ there is a unique pair $a,b$ satisfying the following three conditions 
\begin{equation} \label{3.18}
(1) \quad 1\leq a\leq m-1, \qquad\qquad
(2) \quad 1\leq b\leq n-1, \qquad\qquad
(3) \quad mb+1=na  
\end{equation}
In particular, setting $(c,d):=(m,n)-(a,b)$ we will write, for $m,n>1$,
\begin{equation} \label{3.19}
\spl(m,n):= (a,b)+(c,d).
\end{equation}
Otherwise, we set
\begin{equation} \label{3.20}
a) \quad \spl(1,n):=(1,n-1)+(0,1),
\qquad
b) \quad \spl(m,1):=(1,0)+(m-1,1).
\end{equation}
\end{prop}

All pairs considered being coprime, we are now in a position to give the definition of the  operators $\Qop_{m,n}$ (restricted for the moment to the coprime case) that is most suitable in the present writing.
\begin{defn} \label{defnQ}
For any coprime pair $(m,n)$, we set
\begin{equation} \label{eqdefnQ}
\Qop_{m,n} := 
\begin{cases}
\frac{1}{M}[\Qop_{c,d},\Qop_{a,b}] & \hbox{if }m>1 \hbox{ and } \spl(m,n)=(a,b)+(c,d), \\[6pt]
\D_n & \hbox{if }m=1.
\end{cases}
\end{equation}
\end{defn}
The combinatorial side of the upcoming conjecture is constructed in \cite{Hikita} by Hikita as the Frobenius characteristic of a bi-graded $S_n$ module whose precise definition is not needed in this development. For our purposes  it is sufficient to directly define the \define{Hikita polynomial}, which we denote by $\Hik_{m,n}[\X;q,t]$,  using a process that   closely follows our present rendition of the right hand side of \pref{HHLRU}. That is, we set 
\begin{equation} \label{defHikita}
\Hik_{m,n}[\X;q,t] :=
\sum_{\park \in\Parking_{m,n} } t^{\area(\park) } q^{\dinv(\park)} F_{\ides(\park)}[\X],
\end{equation}
with suitable definitions for all the ingredients occurring in this formula. We will start  
with the collection of $(m,n)$-\define{parking functions}  which we have denoted $\Parking_{m,n}$. Again, a simple example will   suffice. 

%\begin{figure}[ht]
%\begin{center}
%\dessin{height=1.45 in}{pcoarea.pdf}
%\dessin{height=1.45 in}{nicerpath}
%\dessin{height=1.47 in}{armleg.pdf}
%\dessin{height=1.45 in}{park79.pdf}
%\dessin{height=1.45 in}{ranks.pdf}
%\begin{picture}(0,0)(-3,0)\setlength{\unitlength}{5mm}
%\put(-17,7.6){\rotatebox{-15}{\hbox{$\scriptstyle\arm$}}}
%\put(-21,5.1){\rotatebox{30}{\hbox{$\scriptstyle\leg$}}}
%\end{picture}
%\caption{Combinatorial ingredients for the Hikita polynomial.}
%\label{fig:1to5}
%\end{center}
%\end{figure}

\begin{figure}[ht]
\begin{center}\setlength{\unitlength}{5mm}
\dessin{height=2 in}{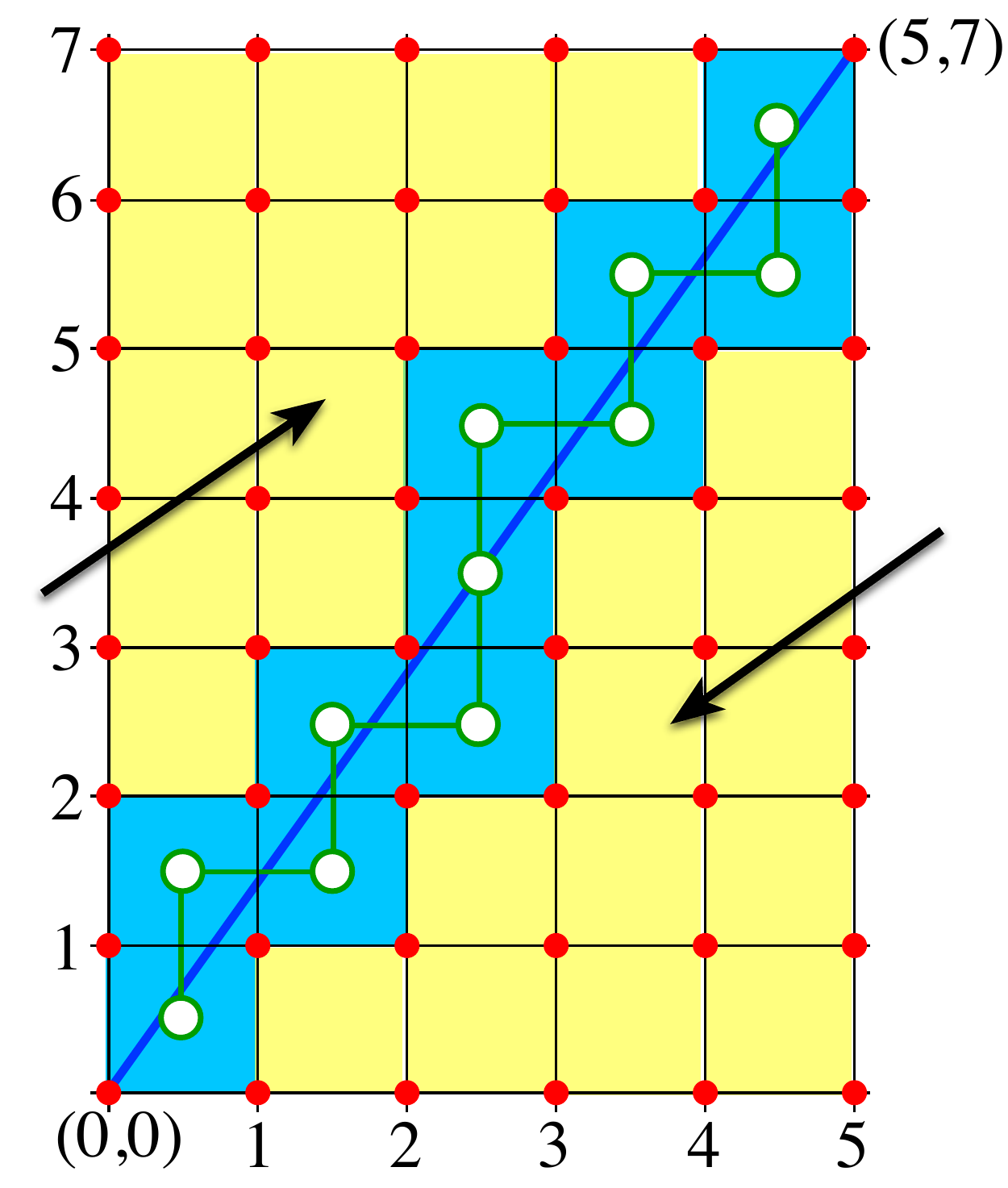}\qquad
\dessin{height=2 in}{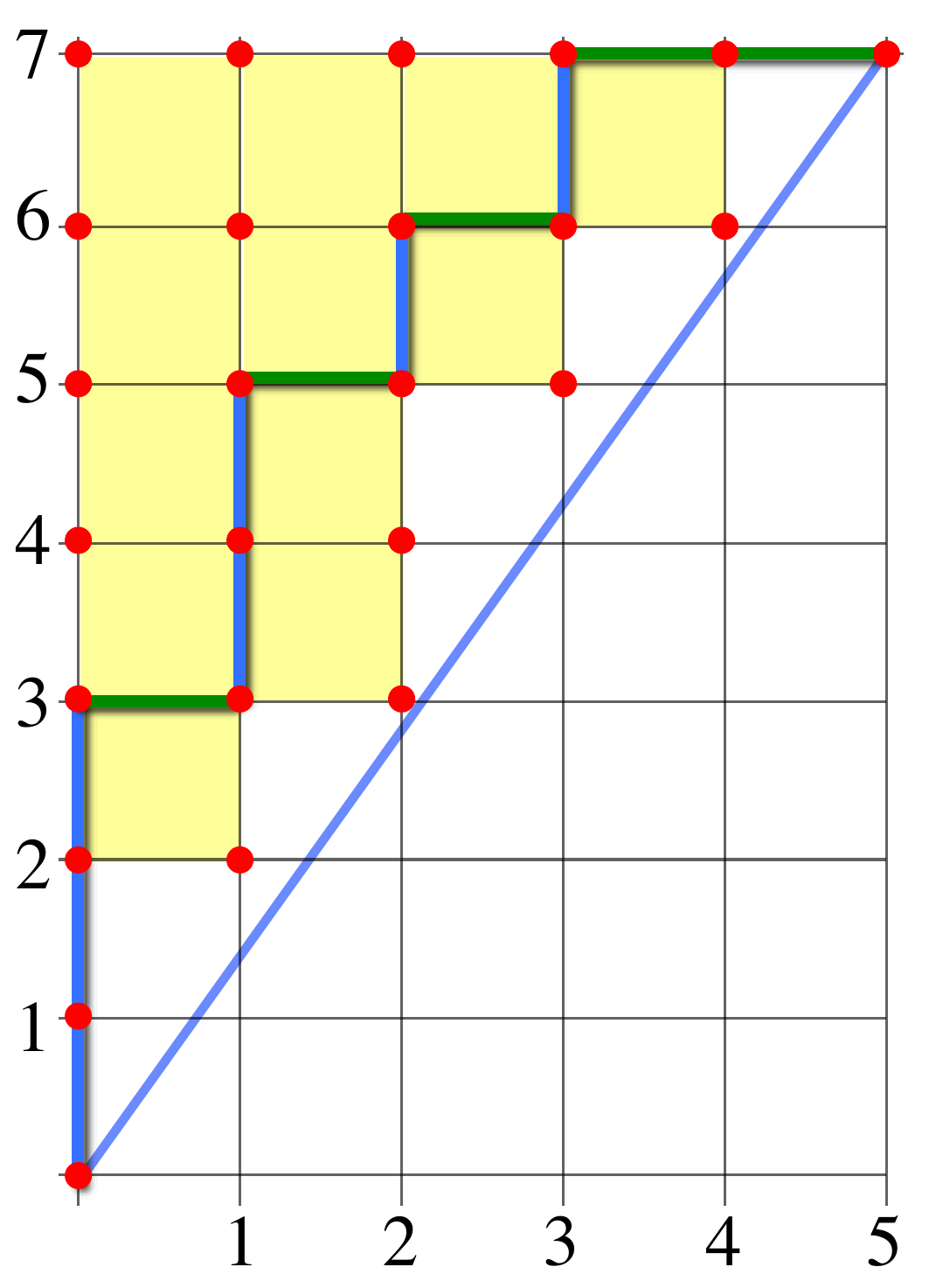}\qquad
\dessin{height=2 in}{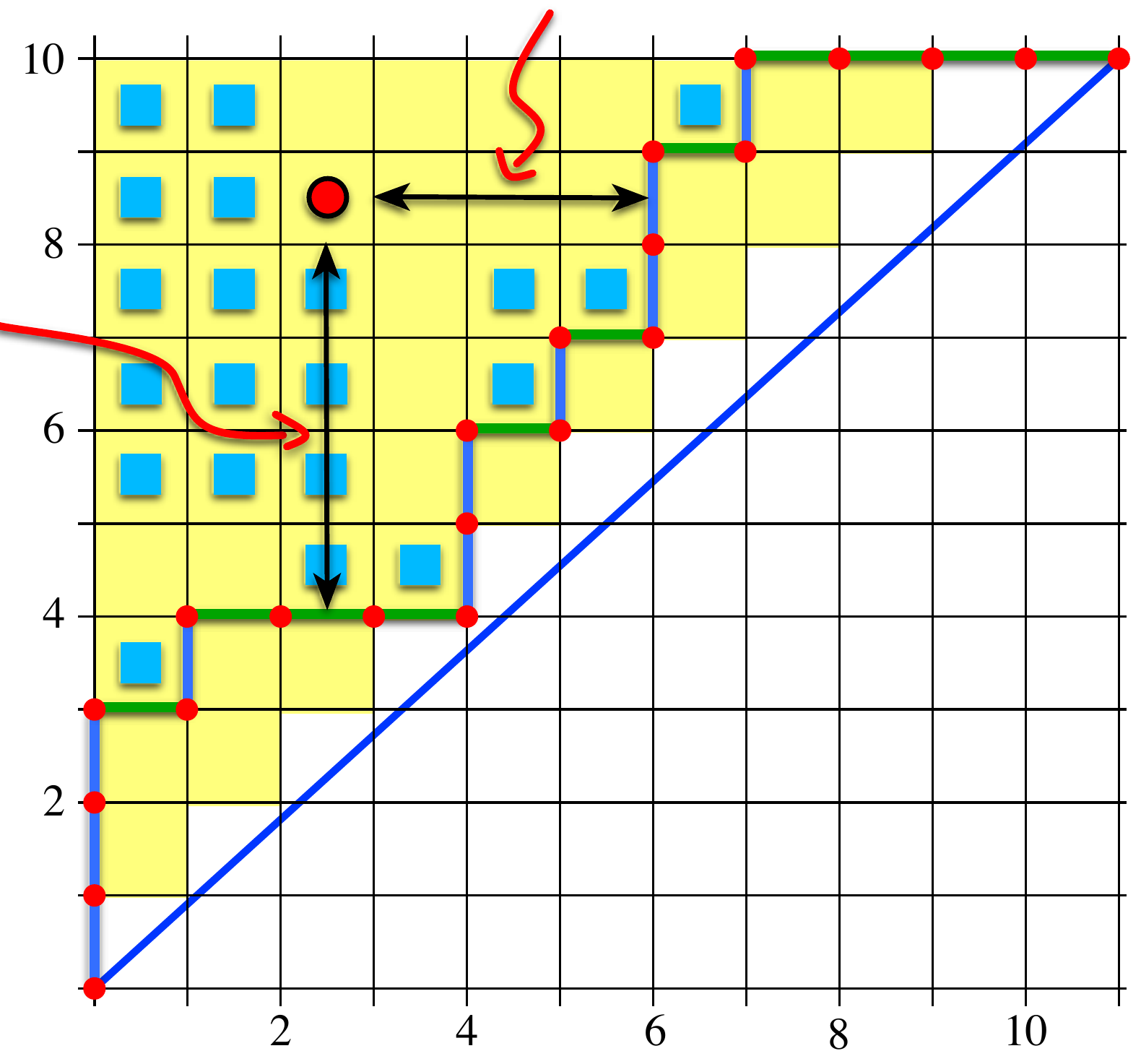}\begin{picture}(0,0)(1,10)
\put(-5,20.5){\rotatebox{-15}{\hbox{$\arm$}}}
\put(-11,17){\rotatebox{30}{\hbox{$\leg$}}}
\end{picture}
\caption{First combinatorial ingredients for the Hikita polynomial.}
\label{fig:1to3}
\end{center}
\end{figure}
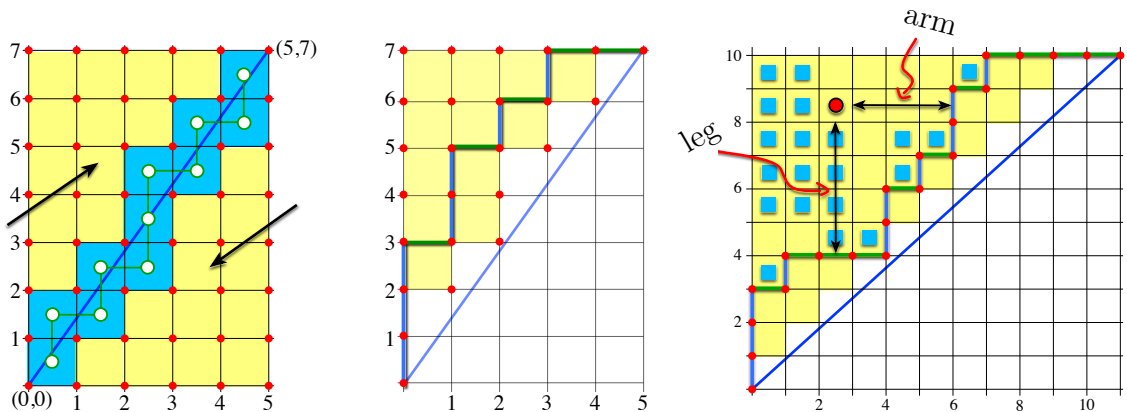

Figures \ref{fig:1to3} and  \ref{fig:4to5} contain all the information needed to construct the polynomial $\Hik_{7,9}[\X;q,t]$. The first object in Figure~\ref{fig:1to3} is a $5 \times 7$ lattice rectangle with its main diagonal $(0,0) \to (5,7)$. In a darker color we have the lattice cells cut by the main diagonal, which we will call the \define{lattice diagonal}. Because of the coprimality of $(m,n)$, the main diagonal, and any line parallel to it, can touch at most a single lattice point inside the $m \times n$ lattice. Thus the main diagonal (except for its end points) remains interior to the lattice cells that it touches. Since the path joining the centers of the touched cells has $n-1$ north steps and $m-1$ east steps, it follows that the lattice diagonal has $m+n-1$ cells. This gives that the number of cells above (or below) the lattice diagonal is $(m-1)(n-1)/2$.

A path in the $m \times n$ lattice that proceeds by north and east steps from $(0,0)$ to $(m,n)$, always remaining weakly above the lattice diagonal, is said to be an $(m,n)$-Dyck path.  For example, the second object in Figure \ref{fig:1to3} is a $(5,7)$-Dyck path. The number of cells between a  path $\path$ and the lattice diagonal is denoted  $\area(\path)$. In the third object of Figure \ref{fig:1to3}, we have an $(11,10)$-Dyck path. Notice that the collection of cells above the path may be viewed as an english Ferrers diagram. We also show there the \define{leg} and the \define{arm} of one of its cells (see Section~\ref{secsymm} for more details). Denoting by $\lambda(\path)$ the Ferrers diagram above the path $\path$, we define
\begin{equation} \label{1.8}
\dinv(\path) := \sum_{c \in \lambda(\path)} \charac\!\left( \frac{\arm(c)}{\leg(c)+1} < \frac{m}{n} < \frac{\arm(c)+1}{\leg(c)}\right).
\end{equation}
As in the classical case an $(m,n)$-parking function is the tableau obtained by labeling the cells east of and adjacent to the north steps of an $(m,n)$-Dyck path with cars $1,2,\dots,n$ in a column-increasing manner. We denote by $\Parking_{m,n}$ the set of $(m,n)$-parking functions. When $(m,n)$ is a pair of coprime integers, it is easy to show that there are $m^{n-1}$ such parking functions. For more on the coprime case, see \cite{armstrong}. We will discuss further aspects of the more general case in \cite{newPleth}, a paper in preparation.
\begin{figure}[ht]
\begin{center}
\dessin{height=2 in}{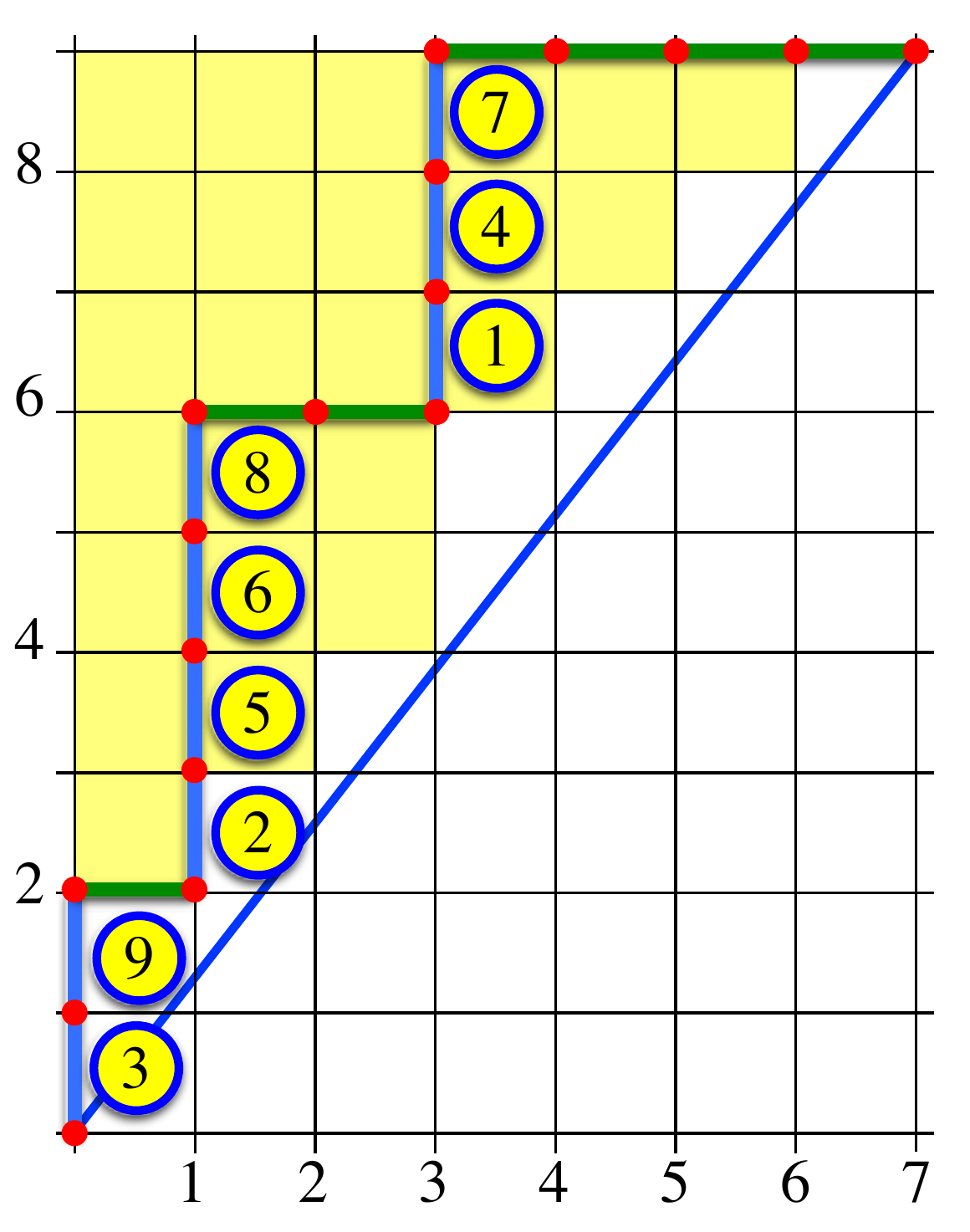}\qquad
\dessin{height=2 in}{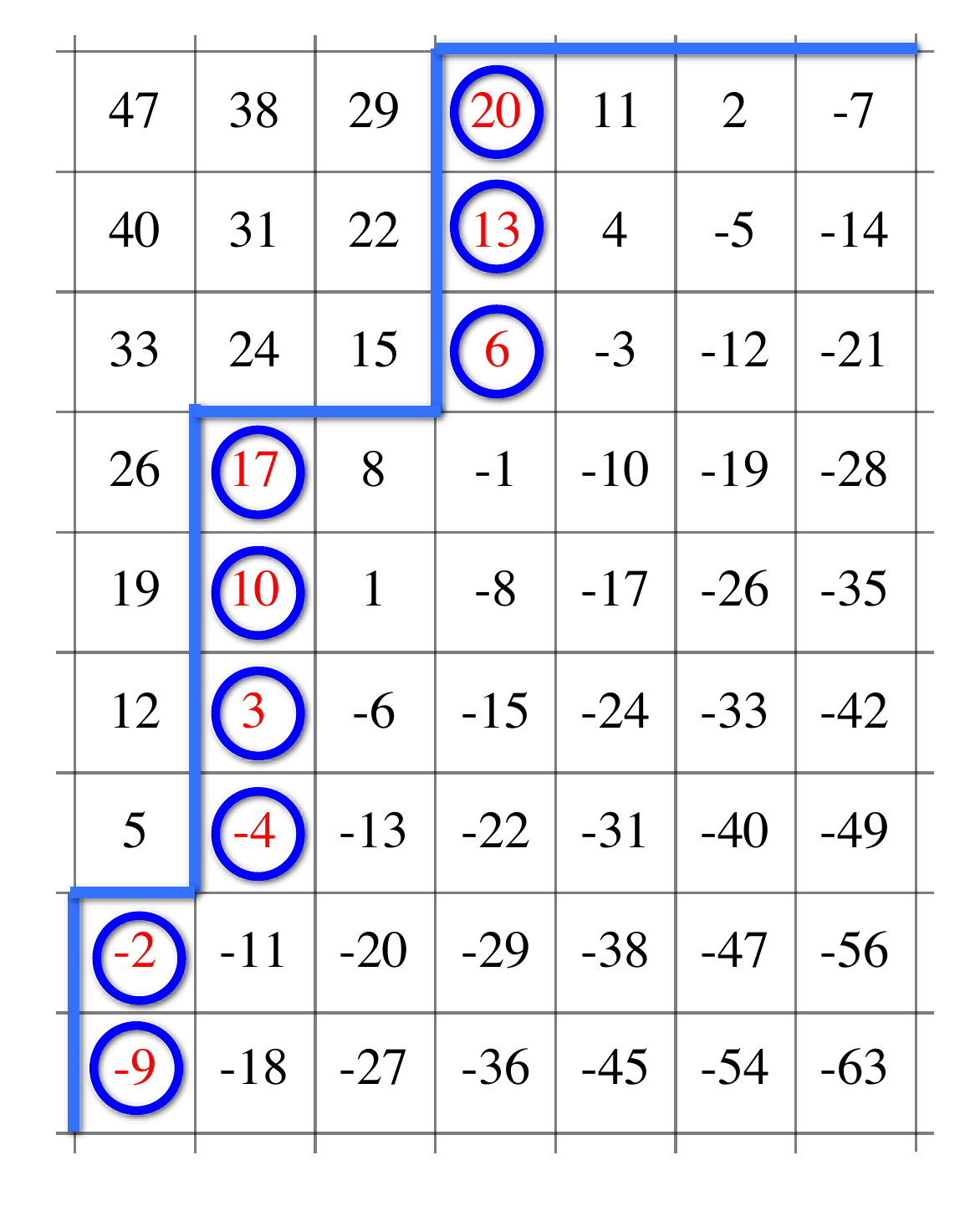}
\caption{Last combinatorial ingredients for the Hikita polynomial.}
\label{fig:4to5}
\end{center}
\end{figure}

The first object in Figure \ref{fig:4to5} gives a $(7,9)$-parking function and the second object gives a $7 \times 9$ table of \define{ranks}. In the general case, this table is obtained by placing in the \define{north-west} corner of the $m\times n$ lattice a number of one's choice. Here we have used $47=(m-1)(n-1)-1$, but the choice is immaterial. We then fill the cells by subtracting $n$ for each east step and adding $m$ for each north step. Denoting by $\rank(i)$ the \define{rank} of the cell that contains car $i$, we define the \define{temporary dinv} of an $(m,n)$-Parking function $\park$ to be the statistic
\begin{equation} \label{1.9}
\tdinv(\park) := \sum_{1\leq i<j \leq n} \charac\!\left( \rank(i)< \rank(j) < \rank(i)+m \right).
\end{equation}

Next let us set for any $(m,n)$-path $\path$
\begin{equation} \label{1.10}
\mdinv(\path) := \max \{ \tdinv(\park) : \path(\park)=\path \},
\end{equation}
where the symbol $\path(\park)=\path$ simply means that $\park$ is obtained by labeling the path $\path$. It will also be convenient to refer to $\path(\park)$ as the \define{support}  of $\park$. This given, we can now set
\begin{equation} \label{1.11}
\dinv(\park) := \dinv(\path(\park)) + \tdinv(\park) - \mdinv(\path(\park)).
\end{equation}
This is a reformulation of Hikita's definition of the \define{dinv} of an $(m,n)$-parking function first introduced by Gorsky and Mazin in \cite{GorskyMazin}.

To complete the construction of the Hikita polynomials we need the notion of the \define{word} of a parking function, which we denote by $\sigma(\park)$. This is the permutation obtained by reading the cars of $\park$ by decreasing ranks. Geometrically, $\sigma(\park)$ can be obtained simply by having a line parallel to the main diagonal sweep the cars from left to right, reading a car the moment the moving line passes through the south end of its adjacent north step. For instance, for the parking function in Figure \ref{fig:4to5} we have $\sigma(\park)= 784615923.$ 

Letting $\ides(\park)$ denote the \define{descent set of the inverse} of the permutation $\sigma(\park)$ and setting, as in the classical case, $\area(\park) = \area(\path(\park))$, we finally have all of the ingredients necessary for \pref{defHikita} to be a complete definition of the Hikita polynomial. The Gorsky-Negut $(m,n)$-Shuffle Conjecture may now be stated as follows.
\begin{conj}[GN-2013] \label{conjNG}
For all coprime pairs of positive integers $(m,n)$, we have
\begin{equation} \label{1.12}
 \Qop_{m,n} \unn = \Hik_{m,n}[\X;q,t].
\end{equation}
\end{conj}
Of course we can use the word ``Shuffle'' again since another use of Gessel's theorem allows us to rewrite \pref{1.12} in the equivalent form
$$
\left\langle \Qop_{m,n}\unn , h_\alpha \right\rangle = \hskip -10pt 
\sum_{\park \in\Parking_{m,n}} \hskip -10pt
t^{\area(\park)} q^{\dinv(\park)} \charac \left( \sigma(\park) \in E_1\shuffle\cdots \shuffle E_k \right),\quad\hbox{for all}\quad \alpha\models n,
$$
where $a_i=|E_i|$ when $\alpha=(a_1,\ldots,a_k)$,  and writing $h_\alpha$ for the product $h_{a_1}\cdots h_{a_k}$.
We must point out that it can be shown that \pref{1.12} reduces to \pref{HHLRU} when $m=n+1$. In fact, it easily follows from the definition in \pref{defnQ} that $\Qop_{n+1,n} = \nabla \D_n \nabla^{-1}.$ This, together with the fact that $\nabla^{-1}\un = 1$ and the definition in \pref{defD}, yields $\Qop_{n+1,n} \unn = \nabla e_n.$ The equality of the right hand sides of \pref{1.12} and \pref{HHLRU} for $m=n+1$ is obtained by a combinatorial argument which is not too difficult.

%%%%%%%%%%%%%%%%%%%%%%%%%%%%%%%%%%
\pagebreak
\section{Our Compositional \texorpdfstring{$(km,kn)$}--Shuffle Conjectures}
The present developments result from theoretical and computer explorations of what takes place in the non-coprime case. Notice first that there is no difficulty in extending the definition of parking functions to the  
$km \times kn$ lattice square, including the $\area$ statistic. Problems arise in extending the definition of the $\dinv$ and $\sigma$ statistics. Previous experience strongly suggested to use the symmetric function side as a guide to the construction of these two statistics. We will soon see that we may remove the coprimality condition in the definition of the $\Qop$ operators, thus allowing us to consider operators  $\Qop_{km,kn}$ which, for $k>1$,  may be simply obtained by bracketing two $\Qop$ operators indexed by coprime pairs. 
However one quickly discovers, by a simple parking function count, that  these $\Qop_{km,kn}$ operators do not provide the desired symmetric function side.

Our search for the natural extension  
of the symmetric function side led us to focus on the following general construction of symmetric function operators indexed by non-coprime pairs $(km,kn)$. 
 This construction is based 
on a simple   commutator  identity satisfied by the operators 
$D_k$ and $D_k^*$ which shows that the $Q_{0,k}$ operator
is none other than multiplication by a  rescaled plethystic version of the ordinary symmetric function $h_k$. 
This implies that
the family 
 $\{\prod_i \Qop_{0,\lambda_i} \}_\lambda$
is a basis for the space of symmetric functions (viewed as multiplication operators).
Our definition also  uses  a  commutativity property 
(proved in [4]) between $\Qop$ operators indexed by collinear vectors, {\em i.e.} $\Qop_{km,kn}$ and $\Qop_{jm,jn}$ commute for all $k$, $j$, $m$ and $n$ (see Theorem~\ref{thmQind}). 
For our purpose, it is convenient to denote by  $\underline{f}$ the operator of \define{multiplication by} $f$ for any symmetric function $f$. 
We can now give our general construction.
\begin{alg} \label{algF}
Given any symmetric function $f$ that is homogeneous of degree $k$, and any coprime pair  $(m,n)$, we  proceed as follows
\begin{enumerate}\itemsep=-4pt
\item[]{\bf Step 1:} calculate the expansion
\begin{equation}
    f = \sum_{\lambda \vdash k} c_\lambda(q,t)\, \prod_{i=1}^{\ell(\lambda)} \Qop_{0,\lambda_i},
\end{equation}
\item[]{\bf Step 2:} using the coefficients $c_\lambda(q,t)$,  set 
\begin{equation} \label{defnF}
\Bf_{km,kn} := \sum_{\lambda \vdash k} c_\lambda(q,t) \prod_{i=1}^{\ell(\lambda)}\Qop_{m \lambda_i, n\lambda_i}.
\end{equation}
\end{enumerate}
\end{alg}

%\begin{figure}[ht]
%\begin{center}
%\includegraphics[height=2.7 in]{multimn.pdf}
%\caption{A $12 \times 20$ Dyck path with composition $(1,2,1)$.}
%\label{multimn}
%\end{center}
%\end{figure}
Theoretical considerations reveal, and extensive computer experimentations confirm, that the operators that we should use to extend  the \define{rational parking function} theory to all pairs $(km,kn)$, are none other than the operators $\Be_{km,kn}$ obtained by taking $f=e_k$  in Algorithm~\ref{algF}. This led us to look for the construction of natural extensions of the definitions of $\dinv(\park)$ and $\sigma(\park)$, that would ensure the validity of the following sequence of increasingly refined conjectures. The coarsest one of which is as follows.%identity 
\begin{conj}\label{conjE} For all coprime pair of positive integers $(m,n)$, and any $k\in\N$, we have
\begin{equation} \label{2.3}
\Be_{km,kn}\cdot {(-\mathbf{1})^{k(n+1)}}= \sum_{\park \in\Parking_{km,kn}}
t^{\area(\park)} q^{\dinv(\park)} F_{\ides(\park)},
\end{equation}
\end{conj}
\begin{wrapfigure}[12]{r}{4cm} \centering
 % \vskip-5pt
     \dessin{height=2.3 in}{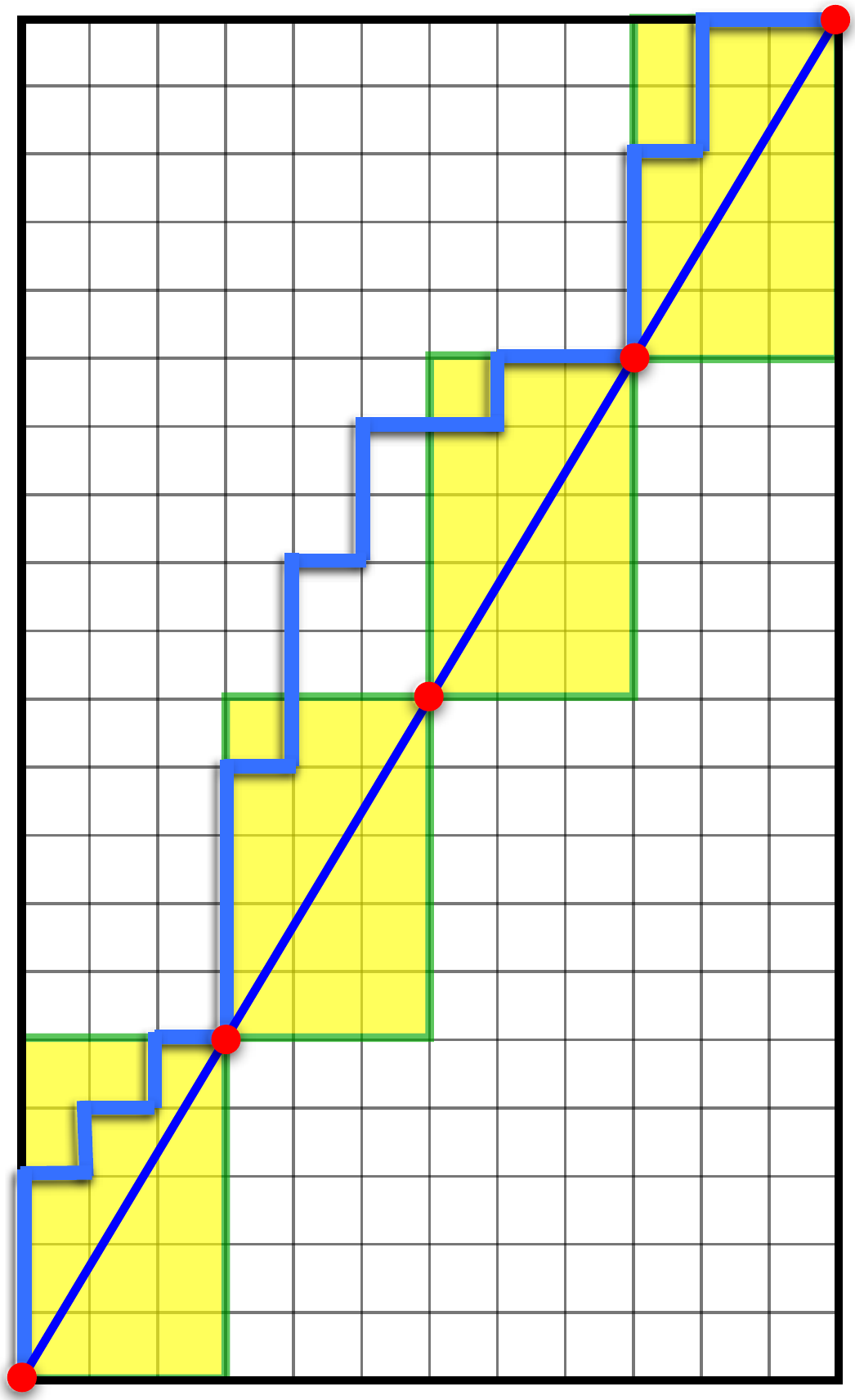}
%\label{multimn}
\end{wrapfigure}
To understand our first refinement, we focus on a special case. In the figure displayed on the right, we have depicted a $12 \times 20$ lattice. The pair in this case has a $\gcd$ of $4$. Thus here $(m,n)=(3,5)$ and $k=4$. Note that in the general case, $(km,kn)$-Dyck paths can hit the diagonal in $k-1$ places within the $km\times kn$ lattice square. In this case, in $3$ places. We have depicted here a Dyck path which hits the diagonal in the first and third places.

At this point the classical decomposition (discovered in \cite{qtCatPos})  
\begin{equation}\label{edecompE}\correction
e_k= E_{1,k}+E_{2,k}+\cdots +E_{k,k},
\end{equation}
combined with extensive computer experimentations, suggested that we have the following refinement of Conjecture~\ref{conjE}.
\begin{conj}\label{conjEr} For all coprime pair of positive integers $(m,n)$, all $k\in\N$, and if $1\leq r\leq k$, we have
\begin{equation} \label{2.4}
 \BE_{km,kn}^{r}\cdot (-\mathbf{1})^{k(n+1)} = \sum_{\park \in\Parking_{km,kn}^{r}}
t^{\area(\park)} q^{\dinv(\park)} F_{\ides(\park)},
\end{equation}
where $\BE_{km,kn}^{r}$ is the operator obtained by setting $f=E_{k,r}$ in Algorithm~\ref{algF}. 
Here $\Parking_{km,kn}^{r}$ denotes the set of parking functions, in the $km \times kn$ lattice, whose Dyck path hit the diagonal in $r$ places {\rm (}including $(0,0)${\rm )}. 
\end{conj}
Clearly, \pref{edecompE} implies that Conjecture~\ref{conjE} follows from Conjecture~\ref{conjEr}. For example, the parking functions supported by the path in the above figure would be picked up by the operator $\BE_{4\times 3,4\times 5 }^{3}$.

Our ultimate refinement is suggested by the decomposition (proved in \cite{classComp})
\begin{equation} \label{2.5}
E_{k,r} = \sum_{\alpha \models k} 
C_{\alpha_1}C_{\alpha_2}\cdots C_{\alpha_r}
 \un.
\end{equation}
What emerges is the following most general conjecture that clearly subsumes our two previous conjectures, as well as Conjectures~\ref{conjHHLRU}, \ref{conjHMZC}, and \ref{conjNG}. 
\begin{conj}[Compositional $(km,kn)$-Shuffle Conjecture] \label{conjBGLX}
For all compositions $\alpha=(a_1,a_2, \dots ,a_r)\models k$ we have 
\begin{equation} \label{2.6}
\BC_{km,kn}^{(\alpha)}\cdot(-\mathbf{1})^{k(n+1)}
= \sum_{\park \in\Parking_{km,kn}^{(\alpha)}}
t^{\area(\park)} q^{\dinv(\park)} F_{\ides(\park)},
\end{equation}
where $\BC_{km,kn}^{(\alpha)}$ is the operator obtained by setting $f=\BC_\alpha\un$ in Algorithm~\ref{algF} and $\Parking_{km,kn}^{(\alpha)}$ denotes the collection of parking functions in the $km\times kn$ lattice whose Dyck path hits the diagonal according to the composition $\alpha$.
\end{conj}
For example, the parking functions supported by the path in the above figure would be picked up by the operator $\BC_{4\times 3 ,4\times 5}^{(1,2,1)}$. 
We will later see that an analogous conjecture may be stated for the operator $\BB_{km,kn}^{(\alpha)}$ obtained by taking $f=\BB_\alpha\un$ in our general Algorithm~\ref{algF}, with $\alpha$ any composition of $k$.  

We  will make extensive use in the sequel of a collection of results stated and perhaps even proved in the works of Schiffmann and Vasserot.
%(see \cite{elliptic2}), \cite{SchiffVassMac}, \cite{SchiffVassK}). 
Unfortunately most of this material is written in a language that is nearly inaccessible to most practitioners of Algebraic Combinatorics. We were fortunate that the two young  researchers E. Gorsky and A. Negut, in a period of several months, made us aware of some of the contents of the latter publications as well as the results in their papers (\cite{GorskyNegut}, \cite{NegutFlags} and \cite{NegutShuffle}) in a language we could understand. The present developments are based on these results. Nevertheless, for sake of completeness we have put together in \cite{newPleth} a purely Algebraic Combinatorial treatment of all the background needed here with proofs that use only the Macdonald polynomial ``tool kit'' derived in the 90's in \cite{SciFi}, \cite{IdPosCon}, \cite{plethMac} and\cite{explicit},  with some additional identities discovered in \cite{HLOpsPF}.

The remainder of this paper is divided into three further sections. In the next section we review some  notation and recall some identities from Symmetric Function Theory, and our Macdonald polynomial tool kit. 
%We then give a precise definition of the $\Qop_{m,n}$ operators for coprime pairs $(m,n)$.  
This done, we state some basic identities that will be instrumental in extending the definition of the $\Qop$ operators to the non-coprime case. 
In the following section we describe how the modular group $\SL_2(\mathbb{Z})$ acts on the operators $\Qop_{m,n}$ and use this action to justify our definition of the operators $\Qop_{km,kn}$. Elementary proofs that justify the uses we make of this action are given in \cite{newPleth}. Here we also  show  how these operators can be efficiently programmed on the computer. This done, we give a precise construction of the operators $\BC_{km,kn}^{(\alpha)}$ and $\BB_{km,kn}^{(\alpha)}$, and workout some examples. We also give a compelling argument which shows the inevitability of Conjecture \ref{conjBGLX}.
In the last section we complete our definitions for all the combinatorial ingredients occurring in the right hand sides of \pref{2.3}, \pref{2.4} and  \pref{2.6}. Finally, we derive some consequences of our conjectures and discuss some possible further extensions.

%%%%%%%%%%%%%%%%%%%%%%%%%%%%%%%%%%

\section{Symmetric function basics, and necessary operators}\label{secsymm}
In dealing with symmetric function identities,  especially those arising in the theory of Macdonald Polynomials, it is convenient and often indispensable to use plethystic notation. This device has a straightforward definition which can be implemented almost verbatim in any computer algebra software. We simply set for any expression $E = E(t_1,t_2 ,\dots )$ and any symmetric function $f$
\begin{equation} \label{3.1}
f[E] := \Qop_f(p_1,p_2, \dots )
\Big|_{p_k \to E( t_1^k,t_2^k,\dots )},
\end{equation}
where $(-)\big|_{p_k \to E( t_1^k,t_2^k,\dots )}$ means that we replace each $p_k$ by $E( t_1^k,t_2^k,\dots )$, for $k\geq 1$.
Here $\Qop_f$ stands for the polynomial yielding the expansion of $f$ in terms of the power basis. We say that we have a \define{plethystic substitution} of $E$ in $f$. 

The above definition of plethystic substitutions implicitly requires  that
$p_k[-E]= -p_k[E]$, and we say that this is the \define{plethystic} minus sign rule. This notwithstanding, we also need to carry out  ``ordinary'' changes of signs. To distinguish the later from the plethystic minus sign, we obtain the \define{ordinary} sign change by multiplying our expressions by a new variable ``$\epsilon$'' which, outside of the plethystic bracket, is replaced by $-1$. Thus we have
\begin{displaymath}
p_k[\epsilon E]= \epsilon^k p_k[ E]= (-1)^k p_k[E].
\end{displaymath}
In particular we see that, with this notation, for any expression $E$ and any symmetric function $f$ we have
\begin{equation} \label{3.2}
(\omega f)[E]= f[-\epsilon E],
\end{equation}
where, as customary, ``$\omega$'' denotes the involution that interchanges the elementary and homogeneous symmetric function bases.
Many symmetric function identities can be considerably simplified by means of the $\Omega$-notation, allied with plethystic calculations. For any expression $E = E(t_1,t_2,\cdots )$ set
\begin{displaymath}
\Omega[E] := \exp\! \left(
\sum_{k \geq 1}{p_k[E]\over k}
\right) = \exp \! \left(
\sum_{k \geq 1} \frac{E(t_1^k,t_2^k,\cdots )}{k} 
\right).
\end{displaymath}
In particular, for $\X=x_1+x_2+\cdots$, we see that
\begin{equation} \label{3.3}
\Omega[z\X]= \sum_{m \geq 0} z^m h_m[\X]
\end{equation}
and for $M=(1-t)(1-q)$ we have
\begin{equation} \label{3.4}
\Omega[-uM] = \frac{(1-u)(1-qtu)}{(1-tu)(1-qu)}.
\end{equation}

\begin{wrapfigure}[7]{r}{4.5cm} \centering
   \vskip-10pt
     \dessin{height=1.2 in}{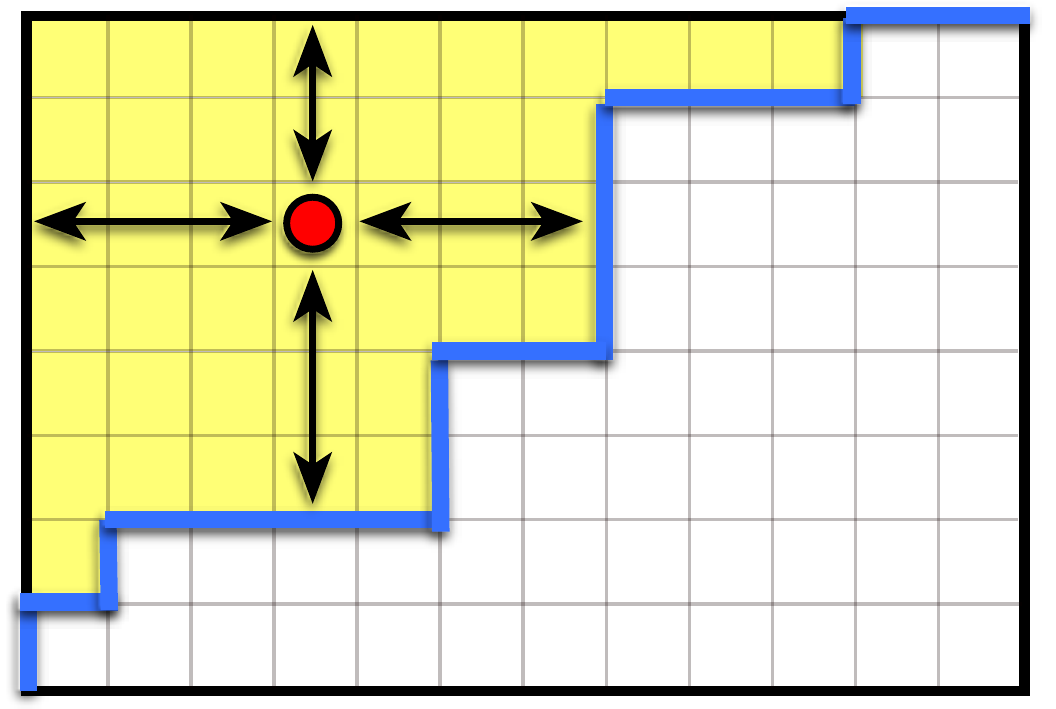}
     \begin{picture}(0,0)(-3,0)\setlength{\unitlength}{3mm}
\put(-2,7.8){$\scriptstyle\arm$}
\put(-7,9){$\scriptstyle\coarm$}
\put(-4.8,6.3){$\scriptstyle\leg$}
\put(-3,10){$\scriptstyle\coleg$}
\end{picture}
\end{wrapfigure}
Drawing the cells of  the Ferrers diagram of a partition $\mu$ as in \cite{Macdonald},  For a cell $c$  in $\mu$, (in symbols  $c\in\mu$), we have parameters $\leg(c)$, and $\arm(c)$, 
which respectively give the number of cells of $\mu$ strictly south of $c$, and strictly east  of $c$. 
Likewise  we have parameters $\coleg(c)$, and $\coarm(c)$, which respectively give the number of cells of $\mu$ strictly north of $c$, and strictly west  of $c$.
This is illustrated in the adjacent figure for the partition that sits above a path.

Denoting by $\mu'$ the conjugate of $\mu$, the basic ingredients we need to keep in mind here are
$$\begin{array}{lll}\displaystyle
\displaystyle n(\mu):=  \sum_{k=1}^{\ell(\mu)} (k-1) \mu_k, \qquad
& \displaystyle T_\mu:=  t^{n(\mu)}q^{n(\mu')},  
\qquad M:=(1-t)(1-q),\\[8pt]
\displaystyle B_\mu(q,t):=  \sum_{c \in \mu} t^{\coleg(c)} q^{\coarm(c)} ,
&\displaystyle \displaystyle\Pi_\mu(q,t):=\prod_{{c\in\mu\atop c\not=(0,0)}} (1-t^{\coleg(c)} q^{\coarm}),
\end{array}$$
and
$$w_\mu(q,t) :=  \prod_{c \in \mu} (q^{\arm(c)} - t^{\leg(c)+1})(t^{\leg(c)} - q^{\arm(c)+1})$$
Let us recall that the Hall scalar product  is defined by setting
\begin{displaymath}
\left\langle p_\lambda, p_\mu \right\rangle\  := \ 
z_\mu \, \chi(\lambda=\mu),
\end{displaymath}
where $z_\mu$ gives the order of the stabilizer of a permutation with cycle structure $\mu$.
The Macdonald polynomials we work with here are the unique (\cite{natBigraded}) symmetric function basis $\{\widetilde{H}_\mu[\X;q,t]\}_\mu$ which is upper-triangularly related (in dominance order) to the modified Schur basis $\{s_\lambda[\frac{\X}{t-1}] \}_\lambda$ and satisfies the orthogonality condition
\begin{equation} \label{3.5}
\left\langle \widetilde{H}_\lambda, \widetilde{H}_\mu \right\rangle_* =\ \charac(\lambda=\mu)\, w_\mu(q,t),
\end{equation}
where $\left\langle-,- \right\rangle_*$ denotes a deformation of the Hall scalar product defined by setting  
\begin{equation} \label{3.6}
\left\langle p_\lambda, p_\mu \right\rangle_*
:=  (-1)^{|\mu|-\ell(\mu)} \prod_i (1-t^{\mu_i})(1-q^{\mu_i})
\, z_\mu \, \charac(\lambda =\mu).
\end{equation}
 We  will use here the operator $\nabla$, introduced in \cite{SciFi}, obtained by setting 
\begin{equation} \label{3.7}
\nabla \widetilde{H}_\mu[\X;q,t]= T_\mu\, \widetilde{H}_\mu[\X;q,t].
\end{equation}
We also set, for any symmetric function $f[\X]$,
\begin{equation} \label{3.8}
\Delta_f \widetilde{H}_\mu[\X;q,t]= f[B_\mu]\, \widetilde{H}_\mu[\X;q,t].
\end{equation}
These families of operators were intensively studied in the $90's$ (see \cite{IdPosCon} and \cite{explicit}) where they gave rise to a variety of conjectures, some of which are still open to this date. In particular it is shown in \cite{explicit} that the operators $\D_k$, $\D_k^*$,  $\nabla$ and the modified Macdonald polynomials $\widetilde{H}_\mu[\X;q,t]$ are related by the following identities.
\begin{equation} \label{formulaoper}
\begin{array}{clcl}
{\rm (i)} & \D_0 \widetilde{H}_\mu = -D_\mu(q,t)\, \widetilde{H}_\mu,
\qquad\qquad &
{\rm (i)}^* & \D_0^* \widetilde{H}_\mu = -D_\mu(1/q,1/t) \widetilde{H}_\mu,
\\[5pt]
{\rm (ii)} & \D_k \ue_1 - \ue_1 \D_k = M \D_{k+1} ,&
{\rm (ii)}^* & \D_k^* \ue_1 - \ue_1 \D_k^* =  -\widetilde{M} \D_{k+1}^*, 
\\[5pt]
{\rm (iii)} & \nabla \ue_1 \nabla^{-1} = -\D_1,  &
{\rm (iii )}^* & \nabla \D_1^* \nabla^{-1} = \ue_1,
\\[5pt]
{\rm (iv)} & \nabla^{-1} \ue_1^\perp \nabla =  {\textstyle \frac{1}{M}} \D_{-1}, &
{\rm (iv )}^* & \nabla^{-1} \D_{-1}^* \nabla = -{\widetilde M}\, \ue_1^\perp,
\end{array}
\end{equation}
with $\ue_1^\perp$ denoting the Hall scalar product adjoint of multiplication by $e_1$, and 
\begin{equation} \label{formDmu}
D_\mu(q,t)= MB_\mu(q,t)-1.
\end{equation}

 We should mention that recursive applications of \pref{formulaoper} ${\rm (ii)}$ and ${\rm (ii)}^*$ give 
\begin{eqnarray} 
  \D_{k} &=&
\frac{1}{M^k} \sum_{i=0}^k {k \choose r}(-1)^r \ue_1^r \D_0 \ue_1^{k-r},\qquad{\rm and}\\
  \D_{k}^* &=&
\frac{1}{\widetilde{M}^k} \sum_{i=0}^k {k \choose r}(-1)^{k-r} \ue_1^r \D_0^* \ue_1^{k-r}.
\end{eqnarray}

For future use, it is convenient to set
\begin{eqnarray} 
 \Phi_k &:=& \nabla \D_k \nabla^{-1} 
\qquad {\rm and}\label{3.12a} \\
  \Psi_k &:=& -(qt)^{1-k} \nabla \D_k^* \nabla^{-1}.\label{3.12b}
\end{eqnarray}
The following identities are then immediate consequences of identities \pref{formulaoper}.  See \cite{newPleth} for details.
\begin{prop} \label{propphipsi}
The operators $\Phi_k$ and $\Psi_k$ are uniquely determined by the recursions
\begin{equation} \label{3.13}
{\rm a)} \quad  \Phi_{k+1} = \frac{1}{M}[ \D_1,\Phi_{k} ]
\qquad {\rm and} \qquad
{\rm b)} \quad  \Psi_{k+1} = \frac{1}{M}[\Psi_{k},\D_1]
\end{equation}
with initial conditions
\begin{equation} 
{\rm a)}\quad  \Phi_1 = \frac{1}{M} [\D_1,\D_0]
\qquad {\rm and} \qquad
{\rm b)}\quad  \Psi_1=-e_1.
\end{equation}
\end{prop}

Next, we must include the following fundamental identity, proved in \cite{newPleth}.
\begin{prop} \label{propDD}
For $a,b \in \mathbb{Z}$ with $n=a+b>0$ and any symmetric function $f[\X]$, we have
\begin{equation} \label{3.15}
\frac{1}{M} ( \D_a \D_b^* - \D_b^* \D_a) f[\X] = \frac{(qt)^b}{qt-1} h_{n} \!\left[\frac{1-qt}{qt}\,\X\right] f[\X].
\end{equation}
\end{prop}

As a corollary we obtain the following.
\begin{prop} \label{propcrochetphipsi}
The operators $\Phi_k$ and $\Psi_k$, defined in \pref{3.12a} and \pref{3.12b}, satisfy the following identity when $a,b$ are any positive integers with sum equal to $n$.
\begin{equation} \label{3.16}
\frac{1}{M} [\Psi_b, \Phi_a] = \frac{qt}{qt-1} \nabla \uh_n \!\left[\frac{1-qt}{qt}\,\X\right] \nabla^{-1}.
\end{equation}
\end{prop}

\begin{proof}[\bf Proof]
The identity in \pref{3.15} essentially says that under the given hypotheses the operator $\frac{1}{M}(\D_b^* \D_a - \D_a \D_b^*)$ acts as multiplication by the symmetric function $\frac{(qt)^b}{qt-1} h_{n} \!\left[(1-qt)\X/(qt)\right]$. Thus, with our notational conventions, \pref{3.15} may be rewritten as
\begin{displaymath}
- \frac{(qt)^{1-b}}{M} \left( \D_b^* \D_a - \D_a \D_b^* \right) = \frac{qt}{qt-1}\ \uh_n\left[\frac{1-qt}{qt}\,\X\right].
\end{displaymath}
Conjugating both sides by $\nabla$, and using 
\pref{3.12a} and \pref{3.12b}, gives \pref{3.16}.
\end{proof}

In the sequel, we will need to keep in mind the following identity which expresses the action of a sequence of $\D_k$ operators on a symmetric function $f[\X]$.
\begin{prop} \label{prop3.2} For all composition $\alpha=(a_1,a_2,\ldots,a_m)$ we have
\begin{equation}
\D_{a_m} \cdots \D_{a_2} \D_{a_1} f[\X] 
= f\!\left[\X+{\textstyle{\sum_{i=1}^m {M}/{z_i}}}\right] \, 
\frac{\Omega[-\mathbf{z}X] }{\mathbf{z}^\alpha} \,
\prod_{1\leq i<j \leq m} \Omega \left[-M \textstyle{z_i/ z_j} \right] \Big|_{\mathbf{z}^0}, 
\end{equation}
where, for $\mathbf{z}=z_1+\ldots +z_m$,  we write $\mathbf{z}^\alpha=z_1^{a_1}\cdots z_m^{a_m}$, and in particular $\mathbf{z}^0=z_1^0z_2^0\cdots z_m^0$.
\end{prop}

\begin{proof}[\bf Proof]
It suffices to see what happens when we use \pref{defD} twice.
\begin{eqnarray*}
\D_{a_2} \D_{a_1}f[\X]
&=& \D_{a_2}f\!\left[\X+{\tfrac{M}{z_1}}\right] \Omega[-z_1\X] \Big|_{z_1^{a_1}}\\
&=& f\!\left[\X+{\textstyle \frac{M}{z_1}}+{\tfrac{M}{z_2}}\right]\,\Omega[-z_1(\X+{\tfrac{M}{z_2}})]\Omega[-z_2\X]\Big|_{z_1^{a_1}z_2^{a_2}} \\ 
&=& f\!\left[\X+{\tfrac{M}{z_1}}+{\tfrac{M}{z_2}}\right]\,\Omega[-z_1\X]\,\Omega[-z_2\X]\, \Omega[-Mz_1/ z_2]
\Big|_{z_1^{a_1}z_2^{a_2}}\\
&=& f\!\left[\X+{\tfrac{M}{z_1}}+{\tfrac{M}{z_2}}\right]\,\frac{\Omega[-(z_1+z_2)\X]}{z_1^{a_1}z_2^{a_2}}\, \Omega[-Mz_1/ z_2]
\Big|_{z_1^0z_2^0},
\end{eqnarray*}
and  the pattern of the general  result clearly emerges.
\end{proof}

%%%%%%%%%%%%%%%%%%%%%%%%%%%%%%%%%%

\section{The \texorpdfstring{$\SL_2[\mathbb{Z}]$}--action and the \texorpdfstring{$ \Qop$}- operators indexed by pairs \texorpdfstring{$(km,kn)$}.}

To extend the definition of the $\Qop$ operators to any non-coprime pairs of indices we need to make use of the action of $\SL_2[ \mathbb{Z} ]$ on the operators $\Qop_{m,n}$. In \cite{newPleth}, $\SL_2[\mathbb{Z}]$ is shown to act on the algebra generated by the $\D_k$ operators by setting, for its generators 
\begin{equation} \label{4.1}
N:=\begin{bmatrix} 1 & 1 \\ 0 & 1 \end{bmatrix}
\qquad {\rm and} \qquad
S:=\begin{bmatrix} 1 & 0 \\ 1 & 1 \end{bmatrix},
\end{equation}

\begin{equation} \label{4.2}
N( \D_{k_1} \D_{k_2} \cdots \D_{k_r} )
= \nabla ( \D_{k_1} \D_{k_2} \cdots \D_{k_r} ) \nabla ^{-1},
\end{equation}
and
\begin{equation} \label{4.3}
S( \D_{k_1} \D_{k_2} \cdots \D_{k_r})
= \D_{k_1+1} \D_{k_2+1} \cdots \D_{k_r+1}.
\end{equation}
It is easily seen that \pref{4.2} is a well-defined action since any polynomial in the $\D_k$ that acts by zero on symmetric functions has an image under $N$  which also acts by zero. In \cite{newPleth}, the same property is shown to hold true for the action of $S$ as defined by \pref{4.3}.

Since $\D_k = \Qop_{1,k}$, and thus $S \Qop_{1,k}= \Qop_{1,k+1}$, it recursively follows that
\begin{equation} \label{4.4}
S\Qop_{m,n}= \Qop_{m,n+m}.
\end{equation}
On the other hand, it turns out that the property
\begin{equation} \label{4.5}
N\Qop_{m,n}=\Qop_{m+n,n},
\end{equation}
is a consequence of the following general result proved in \cite{newPleth}

\begin{prop} \label{prop4.1}
For any coprime pair $m,n$ we have
\begin{equation} \label{4.6}
\Qop_{m+n,n} = \nabla \Qop_{m,n} \nabla^{-1}.
\end{equation}
It then follows from \pref{4.4} and \pref{4.5} that for any $\Big[\begin{matrix} a & c \\[-2pt] b & d \end{matrix}\Big] \in \SL_2[\mathbb{Z}]$, we have
\begin{equation} \label{4.7}
\begin{bmatrix} a & c \\ b & d \end{bmatrix} \Qop_{m,n} = \Qop_{am+cn,bn+dn}.
\end{equation}
\end{prop}

The following identity has a variety of consequences 
in the present development.
\begin{prop} \label{prop4.2}
For any $k \geq 1$ we have $\Qop_{k+1,k}= \Phi_k$ and $\Qop_{k-1,k} = \Psi_k$. In particular, for all pairs $a,b$, of positive integers with sum equal to $n$, it follows that
\begin{equation} \label{4.8}
\frac{1}{M}[\Qop_{b+1,b}, \Qop_{a-1,a}] = 
\frac{qt}{qt-1} \nabla \uh_n \!\left[\frac{1-qt}{qt}\,\X\right] \nabla^{-1}.
\end{equation}
\end{prop}

\begin{proof}[\bf Proof]
In view of  \pref{3.12a} and the second case of \pref{eqdefnQ}, the first equality is a special instance of \pref{4.6}. To prove the second equality, by Proposition \ref{propphipsi}), we only need  to show that the operators $\Qop_{k-1,k}$ satisfy the same recursions and  base cases as the $\Psi_k$ operators. To begin, note that since $\spl(k,k+1)=(1,1)+(k-1,k)$ it follows that
\begin{equation}\label{4.10}
\Qop_{k,k+1} = \frac{1}{M}\left[\Qop_{k-1,k},\Qop_{1,1} \right]
= \frac{1}{M}\left[\Qop_{k-1,k}, \D_1 \right],
\end{equation}
which is (\ref{3.13}b) for $\Qop_{k,k+1}$. However the base case is trivial since by definition $\Qop_{0,1}= -\ue_1$. 
The identity in \pref{4.10} is another way of stating \pref{3.16}.
\end{proof}

\begin{figure}[ht]
\begin{center}
\dessin{height=1.1 in}{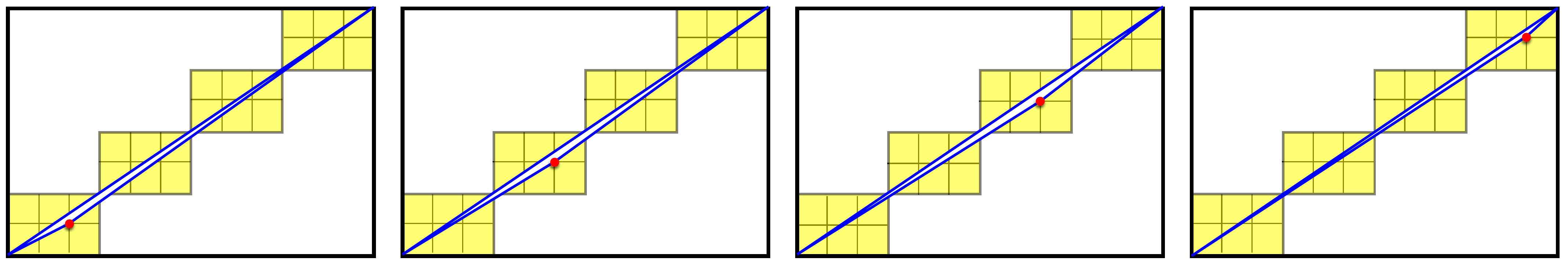}  
\caption{The four splits of (12,8).}
\label{fig:quadrop}
\end{center}
\end{figure}
Our first application is best illustrated by an example.  In Figure \ref{fig:quadrop} we have depicted $k$ versions of the $km\times kn$ rectangle for the case $k=4$ and $(m,n)=(3,2)$. These illustrate that there are $4$ ways to split the vector $(0,0)\to (4\times 3,4\times 2)$ by choosing a closest lattice point below the diagonal. Namely,
\begin{displaymath}
(12,8) = (2,1)+(10,7) = (5,3)+(7,5) = (8,5)+(4,3) = (11,7)+(1,1).
\end{displaymath}
Now, it turns out that the corresponding four bracketings  
\begin{displaymath}
[\Qop_{10,7},\Qop_{2,1}], \qquad 
[\Qop_{7,5},\Qop_{5,3}], \qquad 
[\Qop_{4,3},\Qop_{8,5}], \qquad  {\rm and} \qquad
[\Qop_{1,1},\Qop_{11,7}],
\end{displaymath}
give the same symmetric function operator. This is one of the  consequences of the identity in \pref{4.8}. In fact, the reader should not have any difficulty checking that these four bracketings are the images of the bracketings in \pref{4.8} for $n=4$ by $\big[\begin{matrix}\scriptstyle  2 & \scriptstyle1\\[-5pt] \scriptstyle1 &\scriptstyle 1 \end{matrix}\big]$. Therefore they must also give the same symmetric function operator since our action of $\SL_2[\mathbb{Z}]$ preserves all the identities satisfied by the $\D_k$ operators.

In the general case if $\spl(m,n)=(a,b)+(c,d)$, the $k$ ways are given by
\begin{displaymath}
\left( (u-1)m+a,(u-1)n+b \right) + \left( (k-u)m+c,(k-u)n+d \right),
%\qquad (\hbox{for }1\leq u\leq k).
\end{displaymath}
with $u$ going from $1$ to $k$.

The definition of the $\Qop$ operators in the non-coprime case, as well as some of their remarkable properties, appear in the following result proved in \cite{newPleth}.
\begin{thm} \label{thmQind}
If $\spl(m,n)=(a,b)+(c,d)$ then we may set, for $k> 1$ and any $1\leq u\leq k$,
\begin{displaymath} 
\Qop_{km,kn}= \frac{1}{M}
\left[\Qop_{ (k-u)m+c,(k-u)n+d}, 
\Qop_{(u-1)m+a,(u-1)n+b} \right].
\end{displaymath}
Moreover, letting 
$\A := \Big[\begin{matrix} a & c \\[-2pt] b & d\end{matrix}\Big]$ 
we also have\footnote{Notice $\A \in \SL_2[\mathbb{Z}]$ since (3) of \pref{3.13} gives $ad-bc=1$.}
\begin{displaymath}  
{\rm a)} \quad \Qop_{k,k}= \frac{qt}{qt-1}
\nabla \uh_k \!\left[\frac{1-qt}{qt}\,\X\right] \nabla^{-1},
\qquad {\rm and} \qquad
{\rm b)} \quad \Qop_{km,kn}= \A \Qop_{k,k}.
\end{displaymath}
In particular, it follows that for any fixed $(m,n)$ the operators $\left\{ \Qop_{km,kn} \right\}_{k\geq1}$ form a commuting family.
\end{thm}
 An immediate consequence of Theorem \ref{thmQind} is a very efficient recursive algorithm for computing the action of the operators $\Qop_{km,kn}$ on a symmetric function $f$.
Let us recall that the \define{Lie derivative} of an operator $X$ by an operator $Y$, which we will denote $\delta_Y X$, is simply defined by setting $\delta_Y X = [X,Y]:=XY-YX$. It follows, for instance, that
\begin{displaymath}
\delta_Y^2X= [[X,Y],Y], \qquad \delta_Y^3X= [[[X,Y],Y],Y], \qquad \ldots
\end{displaymath}
Now, our definition gives
\begin{eqnarray*}
   \Qop_{2,1} &=& \frac{1}{M} [\Qop_{1,1},\Qop_{1,0}] =\frac{1}{M}[\D_1,\D_0], \qquad{\rm and}\\
   \Qop_{3,1} &=& \frac{1}{M} [\Qop_{2,1},\Qop_{1,0}] = \frac{1}{M^2}[ [\D_1,\D_0], \D_0],
\end{eqnarray*}
and by induction we obtain
 $$\Qop_{k,1} = \frac{1}{M^{k-1}}\, \delta_{\D_0}^{k-1} \D_1.$$
Thus the action of the matrix $S$ gives
$\Qop_{k,k+1} = \tfrac{1}{M^{k-1}} \delta_{\D_1}^{k-1} \D_2$.
In conclusion we may write
\begin{displaymath}
\Qop_{k,k} = \frac{1}{M}[\Qop_{k-1,k},\Qop_{1,0}]
= \frac{1}{M}[\Qop_{k-1,k},\D_0]
= \frac{1}{M^{k-1}}\left[ \delta_{\D_1}^{k-2} \D_2, \D_0 \right].
\end{displaymath}
This leads to the following recursive general construction of the operator $\Qop_{u,v}$.
\begin{alg} \label{alg1}
Given a pair $(u,v)$ of positive integers:
\begin{enumerate}\itemsep3pt
 \item[] {\bf If} $u=1$ {\bf then} $\Qop_{u,v}:= \D_v$
 \item[]  \qquad {\bf else  if} $u<v$ {\bf then} $\Qop_{u,v}:= S \Qop_{u,v-u}$
 \item[] \qquad {\bf else if} $u>v$, {\bf then} $ \Qop_{u,v}:= N\Qop_{u-v,v}$
 \item[]  \qquad {\bf else}  $\Qop_{u,v}:= \frac{1}{M^{u-1}} \left[ \delta_{\D_1}^{u-2} \D_2, \D_0 \right]$.
 \end{enumerate}
\end{alg}
\noindent
 This assumes that $S$ acts on a polynomial in the $\D$ operators by the replacement $\D_k \mapsto \D_{k+1}$, and $N$ acts by the replacement 
     $$\D_k\   \mapsto\  (-1)^{k}\frac{1}{M^{k}}\delta_{\D_1}^{k} \D_0.$$
   We are now finally in a position to validate our construction (see Algorithm~\ref{algF}) of the operators $\BF_{km,kn}$. To this end, for any partition $\lambda =(\lambda_1,\lambda_2, \dots, \lambda_\ell)$, of length $\ell(\lambda)=\ell$, it is convenient to set
\begin{equation} \label{4.14}
h_\lambda[\X;q,t] = \left(\frac{qt}{qt-1}\right)^\ell
\prod_{i=1}^\ell h_{\lambda_i}\!\left[\frac{1-qt}{qt}\,\X\right].
\end{equation}
Notice that the collection $\left\{ h_\lambda[\X;q,t] \right\}_\lambda$ is a symmetric function basis. Thus we may carry out step one of Algorithm~\ref{algF}.
It may be good to illustrate this in a special case. For instance, when $f=e_3$ we proceed as follows. Note first that for any two expressions $A,B$ we have
\begin{displaymath}
h_3[AB]= \sum_{\lambda \vdash 3} h_\lambda[A]\, m_\lambda[B].
\end{displaymath}
Letting $A=\X(1-qt)/(qt)$ and $B=qt/(qt-1)$ gives
\begin{displaymath}
(-1)^3 e_3[\X] = h_3[-\X] = \sum_{\lambda \vdash 3} h_\lambda[\X;q,t]\, m_\lambda\! \left[ \frac{qt}{qt-1} \right] \left( \textstyle\frac{qt-1}{qt} \right)^{\ell(\lambda)}.
\end{displaymath}
Thus
\begin{equation} \label{4.16}
\Be_{3m,3n} = -\sum_{\lambda \vdash 3}
m_\lambda\! \left[ \frac{qt}{qt-1} \right]
\left(\frac{qt-1}{qt} \right)^{\ell(\lambda)}
\prod_{i=1}^{\ell(\lambda)} \Qop_{m \lambda_i, n \lambda_i}.
\end{equation}
Carrying this out gives  
\begin{displaymath}
\Be_{3m,3n}=
-\frac{1}{[3]_{qt}[2]_{qt}} \Qop_{m,n}^3 
-\frac{qt(2+qt)}{[3]_{qt}[2]_{qt}} \Qop_{m,n}\Qop_{2m,2n}
-\frac{(qt)^2}{[3]_{qt}} \Qop_{3m,3n},
\end{displaymath}
where for convenience we have set $[a]_{qt} = ({1-(qt)^a})/({1-qt})$.

To illustrate, by direct computer assisted calculation, we get
\begin{eqnarray*}
  \Be_{3,6}\cdot(-\mathbf{1}) &=&  s_{{33}}[\X]+ ( {q}^{2}+qt+{t}^{2}+q+t ) s_{{321}}[\X]\\
  &&\quad + ( {q}^{3}+{q}^{2}t+q{t}^{2}+{t}^{3}+qt ) s_{{3111}}[\X]\\
&&\quad + ( {q}^
{3}+{q}^{2}t+q{t}^{2}+{t}^{3}+qt ) s_{{222}}[\X]\\
&&\quad + ( {q}^{4}+{
q}^{3}t+{q}^{2}{t}^{2}+q{t}^{3}+{t}^{4}\\
&&\qquad\qquad+{q}^{3}+2\,{q}^{2}t+2\,q{t}^{2
}+{t}^{3}+{q}^{2}+qt+{t}^{2} ) s_{{2211}}[\X]\\
&&\quad + ( {q}^{5}+{q}
^{4}t+{q}^{3}{t}^{2}+{q}^{2}{t}^{3}+q{t}^{4}+{t}^{5}\\
&&\qquad\qquad+{q}^{4}+2\,{q}^{3
}t+2\,{q}^{2}{t}^{2}+2\,q{t}^{3}+{t}^{4}+{q}^{2}t+q{t}^{2} ) s_{
{21111}}[\X]\\
&&\quad + ( {q}^{6}+{q}^{5}t+{q}^{4}{t}^{2}+{q}^{3}{t}^{3}+{q
}^{2}{t}^{4}+q{t}^{5}+{t}^{6}\\
&&\qquad\qquad+{q}^{4}t+{q}^{3}{t}^{2}+{q}^{2}{t}^{3}+q
{t}^{4}+{q}^{2}{t}^{2} ) s_{{111111}}[\X]
\end{eqnarray*}
Conjecturally, the symmetric polynomial $\Be_{km,kn}\cdot(-\mathbf{1})^{k(n+1)}$ should be the Frobenius characteristic of a bi-graded $S_{kn}$ module. 
In particular the two expressions
\begin{displaymath}
\left\langle \Be_{3,6}\cdot(-\mathbf{1}), e_1^6 \right\rangle
\qquad {\rm and} \qquad
\left\langle \Be_{3,6}\cdot(-\mathbf{1}), s_{1^6} \right\rangle
\end{displaymath}
should respectively give the Hilbert series of the corresponding $S_6$ module and the Hilbert series of its alternants. 
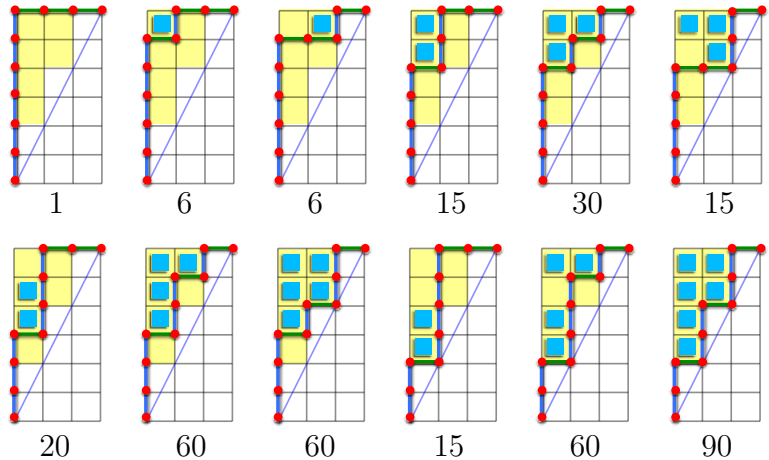
\begin{figure}[ht]
\begin{center}
\dessin{width=4 in}{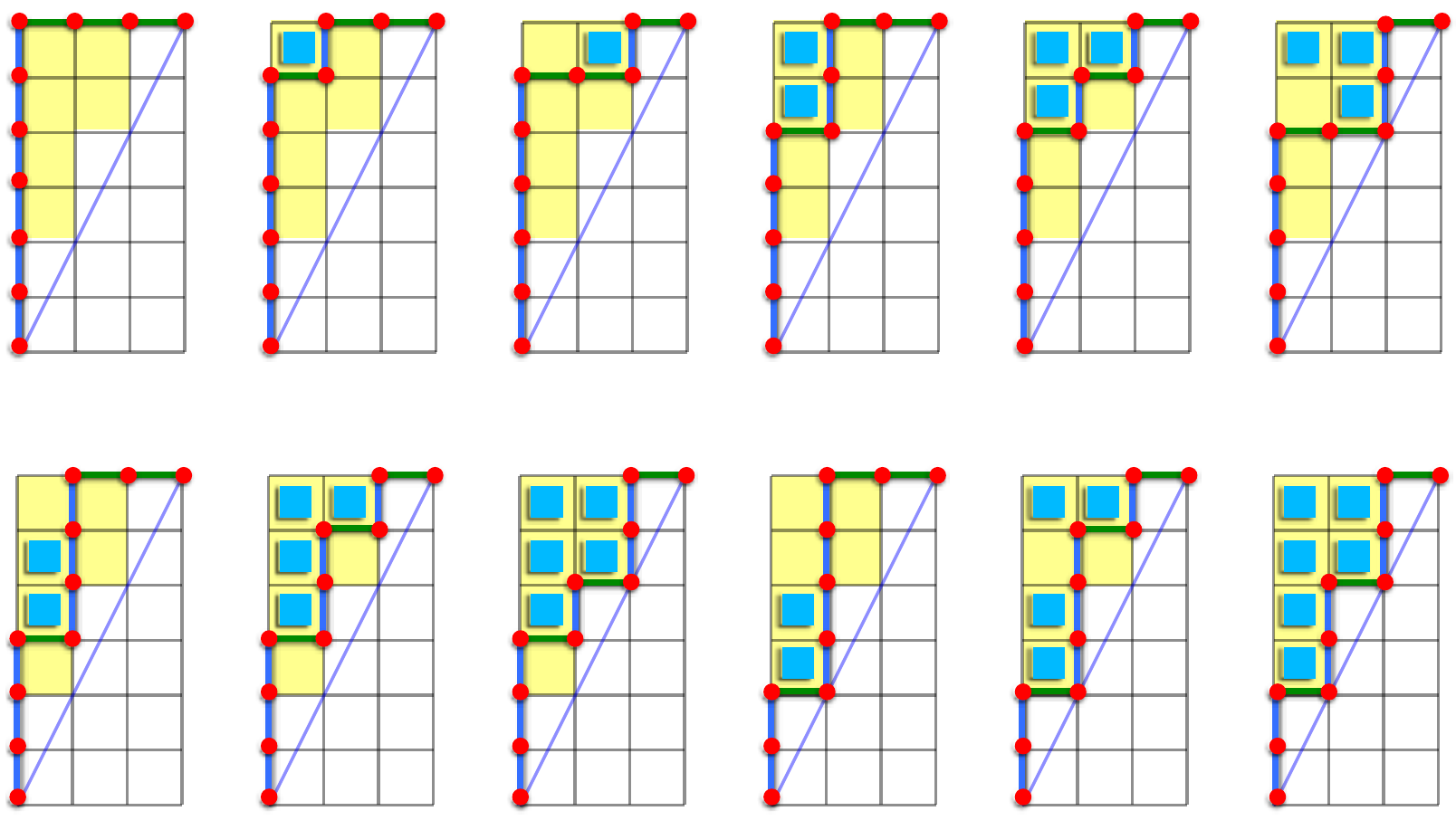}  
\caption{Dyck paths in the $3 \times 6$ lattice.}
\label{fig:ntwelves}
\end{center}
\begin{picture}(0,0)(0,0)\setlength{\unitlength}{6mm}
\put(-7.5,7.4){$1$\hskip15mm $6$\hskip15.3mm  $6$\hskip15mm  $15$\hskip14mm  $30$ \hskip12mm $15$}
\put(-7.7,2){$20$\hskip14mm $60$\hskip13mm  $60$\hskip13mm  $15$\hskip14mm  $60$ \hskip12mm $90$}
\end{picture}
\end{figure}
Our conjecture states that we should have (as in the case of the module of Diagonal Harmonics)
\begin{displaymath}
\left\langle \Be_{3,6}\cdot(-\mathbf{1}), e_1^6 \right\rangle 
= \sum_{\park\in\Parking_{3,6}}
t^{\area(\park)} q^{\dinv(\park)},
\qquad {\rm and} \qquad
\left\langle \Be_{3,6}\cdot(-\mathbf{1}), s_{1^6} \right\rangle 
= \sum_{\path \in \mathcal{D}_{3,6}} t^{\area(\path)}q^{\dinv(\path)},
\end{displaymath}
where the first sum is over all parking functions and the second is over all Dyck paths in the $3 \times 6$ lattice rectangle. 
The first turns out to be the polynomial
\begin{eqnarray*}
   &&{q}^{6}+{q}^{5}t+{q}^{4}{t}^{2}+{q}^{3}{t}^{3}+{q}^{2}{t}^{4}+q{t}^{5}
+{t}^{6}\\
&&\qquad +5\,{q}^{5}+6\,{q}^{4}t+6\,{q}^{3}{t}^{2}+6\,{q}^{2}{t}^{3}+6
\,q{t}^{4}+5\,{t}^{5}\\
&&\qquad +14\,{q}^{4}+19\,{q}^{3}t+20\,{q}^{2}{t}^{2}+19\,
q{t}^{3}+14\,{t}^{4}\\
&&\qquad +24\,{q}^{3}+38\,{q}^{2}t+38\,q{t}^{2}+24\,{t}^{3}
+\\
&&\qquad 25\,{q}^{2}+40\,qt+25\,{t}^{2}+16\,q+16\,t+5.
\end{eqnarray*}
Setting first $q=1$, we get
	$${t}^{6}+6\,{t}^{5}+21\,{t}^{4}+50\,{t}^{3}+90\,{t}^{2}+120\,t+90,$$
which evaluates to $378$ at $t=1$; the second polynomial
\begin{equation} \label{4.17}
 {q}^{6}+{q}^{5}t+{q}^{4}{t}^{2}+{q}^{3}{t}^{3}+{q
}^{2}{t}^{4}+q{t}^{5}+{t}^{6} +{q}^{4}t+{q}^{3}{t}^{2}+{q}^{2}{t}^{3}+q
{t}^{4}+{q}^{2}{t}^{2} ,
\end{equation}
evaluates to $12$, at $q=t=1$. 
All this is beautifully confirmed on the combinatorial side. Indeed, there are 12 Dyck paths in the $3\times 6$ lattice, as presented in Figure \ref{fig:ntwelves}. A simple calculation verifies that the total number of column-increasing labelings of the north steps of these Dyck paths (as recorded in Figure~\ref{fig:ntwelves} below each path) by a permutation of $\{1,2,\dots,6\}$,  is indeed $378$. One may carefully check that this coincidences still holds true when one takes into account the statistics area and dinv.

In the next section we give the construction of the parking function statistics that must be used to obtain the polynomial  $\Be_{3,6}\cdot(-\mathbf{1})$ by purely combinatorial methods. Figure~\ref{fig:ntwelves}  shows the result of a procedure that places a square in each lattice cell, above the path, that contributes a unit to the dinv of that path. Taking into account that the area is the number of lattice squares between the path and the lattice diagonal, the reader should have no difficulty seeing that  the polynomial in \pref{4.17} is indeed the $q,t$-enumerator of the above Dyck paths by $\dinv$ and $\area$. 

\begin{rmk} In  retrospect, our construction of the operators $\BC_{km,kn}^{(\alpha)}$ has a certain degree of inevitability. In fact, since the multiplication operator 
    $$\frac{qt}{qt-1}\, \uh_k\!\left[\frac{1-qt}{qt}\,\X\right]$$
   corresponds to the operator $\Qop_{0,k}$, and since we must have
\begin{displaymath}
\Qop_{k,k} = \nabla \Qop_{0,k} \nabla ^{-1} = \frac{qt}{qt-1} \nabla \uh_k\!\left[\frac{1-qt}{qt}\,\X\right] \nabla ^{-1}
\end{displaymath}
 by Proposition \ref{prop4.1},
then it becomes natural to set 
\begin{displaymath}
\Bf_{k,k} := \nabla \underline{f} \nabla ^{-1}.
\end{displaymath}
for any symmetric function $f$ homogeneous of degree $k$.
Therefore using the matrix $\A$ of Theorem \ref{thmQind}, we obtain
\begin{equation} \label{4.18}
\Bf_{km,kn}= \A \Bf_{k,k}
\end{equation}
In particular, it  follows (choosing $f=e_k$)  that
\begin{displaymath}
\Be_{k,k}\un = \nabla e_k\nabla^{-1} \un= \nabla e_k.
\end{displaymath}
The expansion
$e_k = \sum_{\alpha\models k} C_\alpha \un$
yields the decomposition $\Be_{k,k} = \sum_{\alpha\models k} \BC_{k ,k }^{(\alpha)}$,
and \pref{4.18} then yields
$$
\Be_{km,kn}= \sum_{\alpha\models k}\BC_{km ,kn }^{(\alpha)}.
$$
\qed
\end{rmk}

%%%%%%%%%%%%%%%%%%%%%%%%%%%%%%%%%%

\section{The combinatorial side, further extensions and conjectures.}

Our construction of the parking function statistics in the non-coprime case closely follows what we did in section 1 with appropriate modifications necessary to
resolve conflicts that did not arise in the coprime case. For clarity we will present our definitions as a collection of algorithms which can be directly implemented on a computer.

The symmetric polynomials arising from the right hand sides of our conjectures may also be viewed as   Frobenius characteristics of certain bi-graded $S_n$ modules. Indeed, they are shown to be by \cite{Hikita} in the coprime case. Later in this section we will present some conjectures to this effect.

As for Diagonal Harmonics (see \cite{Shuffle}), all these  Frobenius Characteristics are sums of LLT polynomials. More precisely, a $(km,kn)$-Dyck path  $\path$ may be represented by a vector
\begin{equation} \label{5.1}
\U=(u_1,u_2,\dots,u_{kn}) \qquad
\quad {\rm with}\quad u_0=0, 
\qquad{\rm and}\qquad  u_{i-1} \leq u_i \leq (i-1)\frac{m}{n},
\end{equation}
for all  $2\leq i\leq kn$.
This  given, we set
\begin{equation} \label{5.2}
\LLT(m,n,k;\U) := \sum_{\park \in\Parking(\U)}
t^{\area(\U)} q^{\dinv(\U) + \tdinv(\park) - \mdinv(\U)} s_{\pides(\park)}[\X],
\end{equation}
%\begin{wrapfigure}[6]{r}{4.5cm} \centering
%    \vskip-15pt
%     \dessin{height=1.2 in}{ninesix.pdf}
%     \begin{picture}(0,0)(-3,0)\setlength{\unitlength}{3mm}
%\put(-4.2,10.3){$\scriptstyle\arm$}
%\put(-7.1,7.3){$\scriptstyle\leg$}
%\end{picture}
%\end{wrapfigure}
where $\Parking(\U)$ denotes the collection of parking functions supported by the path corresponding to $\U$. We also use here the Egge-Loehr-Warrington (see \cite{expansions}) result and substitute the Gessel fundamental by the Schur function indexed by $\pides(\park)$ (the descent composition of the inverse of the permutation $\sigma(\park)$). 
Here $\sigma(\park)$ and the other statistics used in \pref{5.2} are constructed according to the following algorithm.

%\begin{figure}[ht]
%\begin{center}
%\includegraphics[height=1 in]{ninesix.pdf}  
%\caption{A $12 \times 8$-Dyck path.}
%\label{fig:ninesix}
%\end{center}
%\end{figure}

\begin{alg} \label{algL}
 \begin{enumerate}
 
\item Construct the collection $\Parking(\U)$ of vectors $\V=( v_1, v_2, \dots, v_{kn} )$ which are the permutations of $1, 2, \dots, kn$ that satisfy the conditions
\begin{displaymath}
u_{i-1}=u_i \quad\implies\quad v_{i-1}<v_{i}.
\end{displaymath} 
\item Compute the $\area$ of the path, that is the number of full cells between the path and the
main diagonal of the $km\times kn$ rectangle, by the formula
\begin{displaymath}
\area(\U)= (kmkn - km - kn + k)/2 - \sum_{i=1}^{kn} u_i .
\end{displaymath}
\item Denoting by $\lambda(\U)$ the partition above the path, set% {\rm (}see figure above{\rm )}, 
%for each cell $c \in \lambda(\U)$ let $a(c)$ and $\ell(c)$ denote the arm and leg of $c$ in $\lambda(\U)$.  
\begin{displaymath}
\dinv(\U) = \sum_{c \in \lambda(\U)} \charac\!\left( \frac{\arm(c)}{\leg(c)+1} \le \frac{m}{n} < \frac{\arm(c)+1}{\leg(c)}\right).
\end{displaymath}
\item Define the $\rank$ of the $i^{th}$ north step by $km(i-1) - kn u_i + u_i/(km+1)$ and accordingly use this number as the $\rank$ of car $v_i$, which we will denote as $\rank(v_i)$. This given, we set, for $\park = \Big[\begin{matrix} \U \\[-3pt] \V \end{matrix}\Big]$
\begin{displaymath}
\tdinv(\park) = \sum_{1\leq r<s \leq kn}
\charac\left( \rank(r) < \rank(s) < \rank(r)+km \right).
\end{displaymath}
\item Define $\sigma(\park)$ as the permutation of $1, 2, \dots, kn$ obtained by reading the cars by decreasing ranks. Let $\pides(\park)$ be the composition whose first $kn-1$ partial sums give the descent set of the inverse of $\sigma(\park)$.
\item Finally $\mdinv(\U)$ may be computed as $\max \{\tdinv(\park) : \park \in\Parking(\U) \}$, or more efficiently as the $\tdinv$ of the $\park$ whose word $\sigma(\park)$ is the reverse permutation $(kn)\cdots 3 2 1$.
\end{enumerate}
\end{alg}
 This completes our definition of the polynomial $\LLT(m,n,k;\U)$, which may be shown to expand as a linear combination of Schur functions with coefficients in $\N[q,t]$. It may be also be shown that, at $q=1$, the polynomial $\LLT(m,n,k;\U)$ specializes to $t^{\area(\U)}e_{\lambda(\U)}[\X]$, with $\lambda(\U)$ the partition giving the multiplicities of the components of $\U_{km,kn}-\U$, with $\U_{km,kn}$ the vector encoding the unique $0$-area $(km,kn)$-Dyck path.
 
 With the above definition at hand, Conjecture \ref{conjBGLX} can be restated as
\begin{equation} \label{5.3}
\BC^{(\alpha)}_{km,kn}\cdot (-\mathbf{1})^{k(n+1)}
=\sum_{\U \in \mathcal{U}(\alpha)}\LLT(m,n,k;\U)
\end{equation}
where $\mathcal{U}(\alpha)$ denotes the collection of all $\U$ vectors satisfying \pref{5.1} whose corresponding Dyck path hits the diagonal of the $km\times kn$ rectangle according to the composition $\alpha\models k$. Alternatively we can simply require that $\alpha$ be the composition of $k$ corresponding to the subset
\begin{displaymath}
\left\{ 1\leq i\leq k-1 : u_{ni+1}=i m \right\}.
\end{displaymath}
Note that, although the operators $\BB_b$ and $ \sum_{\beta \models b}
\BC_{\beta}$ are different, in view of definition \pref{defC} they take the same value on constant symmetric functions
\begin{equation} \label{5.4}
\BB_b \un = e_b[\X] = \sum_{\beta \models b}
\BC_{\beta} \un.
\end{equation}
This circumstance, combined with the commutativity  relation (proved in \cite{classComp}) 
\begin{equation} \label{5.5}
\BB_b \BC_\gamma= q^{\ell(\gamma)} \BC_\gamma \BB_b, 
\end{equation}
for all compositions $\gamma$,
enables us to derive a $(km, kn)$ version of the 
Haglund-Morse-Zabrocki conjecture \ref{conjHMZB}. To see how this comes about we briefly reproduce an argument first given in \cite{classComp}. 
Exploiting \pref{5.4} and \pref{5.5}, we calculate that
\begin{align*}
\BB_a \BB_b \BB_c \un &= \BB_a \BB_b \sum_{\gamma \models c} \BC_{\gamma} \un \\ 
&= \BB_a \sum_{\gamma \models c}q^{\ell(\gamma)} \BC_{\gamma} \BB_b \un \\ 
&= \BB_a \sum_{\gamma\models c} q^{\ell(\gamma)} \BC_{\gamma}\sum_{\beta \models b} \BC_{\beta} \un \\ 
&= \sum_{\gamma\models c} q^{2 \ell(\gamma)} \BC_{\gamma}\sum_{\beta\models b} q^{\ell(\beta)} \BC_{\beta}\sum_{\alpha \models a} \BC_{\alpha} \un.
\end{align*}
From this example one can easily derive that the following  identity holds true in full generality.
\begin{prop} \label{5.1}
For any $\beta=(\beta_1,\beta_2,\dots ,\beta_k)$, we have
\begin{equation} \label{5.6}
\BB_{\beta} \un
= \sum_{\alpha \preceq \beta} q^{c(\alpha,\beta)}
\BC_{\alpha} \un
\end{equation}
where $\alpha \preceq \beta$ here means that $\alpha$ is a refinement of the reverse of $\beta$. That is $\alpha = \alpha^{(k)} \cdots \alpha^{(2)}\alpha^{(1)}$ with $\alpha^{(i)} \models \beta_{i}$, and in that case
\begin{displaymath}
c(\alpha,\beta) = \sum_{i=1}^k (i-1)\, \ell(\alpha^{(i)}).
\end{displaymath}
\end{prop}

Using \pref{5.6} we can easily derive the following.
\begin{prop} Assuming that Conjecture~\ref{conjBGLX} holds, then
for all compositions 
   $$\beta=(\beta_1,\beta_2,\dots ,\beta_{\ell)}\models k$$ we have 
\begin{displaymath}
\BB_{km,kn}^{(\beta)}\cdot (-\mathbf{1})^{k(n+1)}
= \sum_{\alpha \preceq \beta} q^{c(\alpha,\beta)}
\sum_{\park \in\Parking_{km,kn}^{(\alpha)}}
t^{\area(\park)} q^{\dinv(\park)} F_{\ides(\park)}.
\end{displaymath}
Here $\BB_{km,kn}^{(\beta)}$ is the operator obtained by setting $f=\BB_{\beta} \un$ in Algorithm~\ref{algF} and, as before, $\Parking_{km,kn}^{(\alpha)}$ denotes the collection of parking functions in the $km\times kn$ lattice whose Dyck path hits the diagonal according to the composition $\alpha$.
\end{prop}

A natural question that arises next is what parking function interpretation may be given to the 
polynomials  $\Qop_{km,kn} \cdot(-\mathbf{1})^{k(n+1)}$. Our attempts to answer this question lead to a variety of interesting identities. We start with a known identity    and two new ones.

\begin{thm} \label{thmQsquare} For all $n$, we have
\begin{eqnarray} 
\Qop_{n,n+1} \unn &=& \nabla e_n,\qquad {\rm and}\label{5.7a}\\
 \Qop_{n,n} \unn &=& (-qt)^{1-n}\nabla \Delta_{e_1} h_n=\Delta_{e_{n-1}} e_n.\label{5.7b}
\end{eqnarray}
\end{thm}

\begin{proof}[\bf Proof]
Since $\Qop_{n+1,n} = \nabla \Qop_{1,n}\nabla ^{-1}$ and $\Qop_{1,n} = \D_n$, then
\begin{eqnarray*}
\Qop_{n,n+1} \unn &=& \nabla \D_n \nabla ^{-1} \unn\\
     &=&  \nabla \D_n \unn. 
\end{eqnarray*}
Recalling that  $\nabla^{-1} \, \unn =(-1)^{n}$ and $\D_n \un = (-1)^n e_n$, this gives \pref{5.7a}. The proof of \pref{5.7b} is a bit more laborious.  We will obtain it below, by combining a few auxiliary identities.
\end{proof}

\begin{prop} \label{prop5.2}
For any monomial $m$ and $\lambda \vdash n$
\begin{equation} \label{5.8}
s_\lambda[1-m] = 
\begin{cases}
(-m)^k (1-m) & \hbox{\rm if}\  \lambda = (n-k,1^k)\ \hbox{\rm for some }\ 0\leq k\leq n-1,\\[4pt]
0 &  \hbox{\rm otherwise.}
\end{cases}
\end{equation}
\end{prop}
\noindent
This is an easy consequence of the addition formula for Schur functions.

\begin{prop}\label{prop5.3} For all $n$,
\begin{equation} \label{5.9}
\frac{qt}{qt-1} h_n\! \left[ \frac{1-qt}{qt}\,\X\right] 
= -(qt)^{1-n} \sum_{k=0}^{n-1} (-qt)^k s_{n-k,1^k}[\X].
\end{equation}
\end{prop}

\begin{proof}[\bf Proof]
The Cauchy formula gives
\begin{displaymath}
\frac{qt}{qt-1} h_n\!\left[\frac{1-qt}{qt}\,\X\right]
= \frac{(qt)^{1-n}}{qt-1} \sum_{\lambda \vdash n} s_\lambda[\X] s_\lambda[1-qt] 
\end{displaymath}
and \pref{5.8} with $m=qt$ proves \pref{5.9}.
\end{proof}
\noindent
Observe that when we set $qt=1$, the right hand side of \pref{5.9} specializes to $-p_n$.

\begin{prop} \label{prop5.4} For all $n$,
\begin{equation} \label{5.10}
\Delta_{e_1} h_n[\X] = \sum_{k=0}^{n-1} (-qt)^k\,s_{n-k,1^k}[\X].
\end{equation}
\end{prop}

\begin{proof}[\bf Proof]
By \pref{3.8} and \pref{formulaoper} $(i)$ we have $\Delta_{e_1} = (I-\D_0)/M$. Thus, by \pref{defD}
\begin{eqnarray*}
M \Delta_{e_1} h_n[\X]
&=&  h_n[\X] - h_n[\X+{M}/{z}]\, \Omega[-z\X]\Big|_{z^0}\\ 
&=& - \sum_{k=1}^{n} (-1)^k h_{n-k}[\X] \, e_k[\X]\, h_k[M] \\ 
&=& - \sum_{k=1}^{n-1} (-1)^k s_{n-k,1^k}[\X]\, h_k[M] + \sum_{k=0}^{n-1} (-1)^{k} s_{n-k,1^k}[\X]\, h_{k+1}[M] \\ 
&=& \sum_{k=1}^{n-1} (-1)^k s_{n-k,1^k}[\X]\, (h_{k+1}[M]-h_{k}[M]) + s_n[\X] M.
\end{eqnarray*}
This proves \pref{5.10} since the Cauchy formula and \pref{5.8} give
\begin{eqnarray*}
h_n[M] &=& \sum_{k=0}^{n-1}(-t)^k(1-t)(-q)^k(1-q)\\
           &=& M\sum_{k=0}^{n-1}(qt)^k.
\end{eqnarray*}
\end{proof}
Now, it is shown in \cite{explicit} that $e_n[\X]$ and $h_n[\X]$ have the expansions
\begin{align}
e_n[\X] &= \sum_{\mu \vdash n}\frac{M\, B_\mu(q,t) \Pi_\mu(q,t)}{w_\mu}\,\widetilde{H}_\mu[\X;q,t], \qquad {\rm and} \label{5.11a}\\ 
h_n[\X] &= (-qt)^{n-1} \sum_{\mu\vdash n}\frac{M\,B_\mu(1/q,1/t)\, \Pi_\mu(q,t)}{w_\mu}\,\widetilde{H}_\mu[\X;q,t]. \label{5.11b}
\end{align}
We are then ready to proceed with the rest of our proof.
\begin{proof}[\bf Proof][Proof of Theorem \ref{thmQsquare} continued]
The combination of \pref{5.9}, \pref{5.10} and 
\pref{5.11a} gives
\begin{eqnarray} 
  \frac{qt}{qt-1} h_n\!\left[\frac{1-qt}{qt}\,\X\right] 
            &=& -(qt)^{1-n} \Delta_{e_1} h_n\label{5.12}\\
             &=& (-1)^{n} \sum_{\mu\vdash n} \frac{ M B_\mu(q,t) B_\mu(1/q,1/t)\, \Pi_\mu(q,t)}{w_\mu}\,\widetilde{H}_\mu[\X;q,t] \label{5.13}
\end{eqnarray}
Now from Theorem \ref{thmQind} we derive that 
\begin{eqnarray*}
\Qop_{n,n}\unn &=& \frac{qt}{qt-1} \nabla \uh_n \!\left[\frac{1-qt}{qt}\,\X\right] \nabla^{-1} \unn\\
                         &=& (-1)^n \frac{qt}{qt-1} \nabla h_n\!\left[\frac{1-qt}{qt}\,\X\right].
\end{eqnarray*}
Thus  \pref{5.12} gives
\begin{displaymath}
\Qop_{n,n} \unn = (-qt)^{1-n} \nabla \Delta_{e_1}h_n.
\end{displaymath}
This proves the first equality in \pref{5.7b}. The second equality in \pref{5.13} gives
\begin{equation} \label{5.13}
\Qop_{n,n} \unn = \sum_{\mu\vdash n}\frac{T_\mu\,  M B_\mu(q,t) B_\mu(1/q,1/t)\, \Pi_\mu(q,t)}{w_\mu}\,\widetilde{H}_\mu[\X;q,t].
\end{equation}
But it is not difficult to see that we may write (for any $\mu\vdash n$)
\begin{displaymath}
T_\mu \,B_\mu(1/q,1/t)= e_{n-1}\left[B_\mu(q,t)\right]
\end{displaymath}
and \pref{5.13} becomes (using  \pref{5.11a})
\begin{eqnarray*}
\Qop_{n,n} \unn &=& \Delta_{e_{n-1}} \sum_{\mu\vdash n} \frac{ M B_\mu(q,t)\, \Pi_\mu(q,t)}{w_\mu}\,\widetilde{H}_\mu[\X;q,t]\\
                          &=& \Delta_{e_{n-1}} e_n.
\end{eqnarray*}
\end{proof}
To obtain a combinatorial version of \pref{5.7b} we need some auxiliary facts.
\begin{prop}  For all positive integers $a$ and $b$, we have
\begin{equation} \label{5.15}
\BC_a \BB_b \un = \BC_a e_b[\X] =  (-1/q)^{a-1} s_{a,1^{b}}[\X]  - (-1/q)^a s_{1+a,1^{b-1}}[\X],  
\end{equation}
and
\begin{equation} \label{5.16}
\sum_{\beta\models n-a} \BC_a \BC_\beta \un
= (-1/q)^{a-1} s_{a,1^{n-a}}[\X]  - (-1/q)^a s_{1+a,1^{n-a-1}}[\X].
\end{equation}
\end{prop}

\begin{proof}[\bf Proof]
The equality in \pref{5.4} gives the first equality in \pref{5.15} and the equivalence of \pref{5.15} to \pref{5.16}, for $b=n-a$. Using \pref{defC} and \pref{5.8} we derive that
\begin{eqnarray*}
\BC_a e_b[\X]
&=& (-1/q)^{a-1} \sum_{r=0}^b e_{b-r}[\X] (-1)^r\, h_r \left[ (1-1/q)/z \right]\, \Omega[z\X] \Big|_{z^a} \\ 
&=&  (-1/{q})^{a-1} \left(e_b[\X]\, h_a[\X] + ( 1- 1/q) \sum_{r=1}^b (-1)^r e_{b-r}[\X]\, h_{r+a}[\X] \right) \\ 
&=& (-1/q)^{a-1} \Big(s_{a,1^{b}}[\X]+s_{a+1,1^{b-1}}[\X] +  \\ 
&& \qquad\qquad  + (1-1/q) \left( \textstyle \sum_{r=1}^b (-1)^r  s_{r+a,1^{b-r}}[\X] -
\sum_{r=2}^{b} (-1)^r s_{r+a ,1^{b-r}}[\X] \right) \Big) \\ 
&=& (-1/q)^{a-1} \left(s_{a,1^{b}}[\X]+s_{a+1,1^{b-1}}[\X] - (1-1/q)\, s_{1+a,1^{b-1}}[\X] \right) \\ 
&=& (-1/q)^{a-1} s_{a,1^{b}}[\X]-(-1/q)^a s_{1+a,1^{b-1}}[\X].
\end{eqnarray*}
\end{proof}

The following conjectured identity is well known and is also stated in \cite{Shuffle}. We will derive it from Conjecture \ref{conjHMZC} for sake of completeness.

\begin{thm} \label{thm5.3}
Upon the validity of the Compositional Shuffle conjecture we have 
\begin{equation} \label{5.17}
\nabla (-q)^{1-a} s_{a,1^{n-a}} = 
\sum_{\park \in\Parking_{n, \geq a}}
t^{\area(\park)} q^{\dinv(\park)} F_{\ides(\sigma(\park))},
\end{equation}
where the symbol $\Parking_{n, \geq a}$ signifies that the sum is to be extended over the parking functions in the $n\times n$ lattice whose Dyck path's first return to the main diagonal occurs in a row    $y \geq a$.
\end{thm}

\begin{proof}[\bf Proof]
An application of $\nabla$ to both sides of \pref{5.16} yields   
\begin{equation} \label{5.18}
\sum_{\beta\models n-a} \nabla \BC_a \BC _\beta\un
= ( \tfrac{1}{q})^{a-1} \nabla (- 1)^{a-1} s_{a,1^{n-a}} - ( \tfrac{1}{q})^{a} \nabla (-1)^{a} s_{1+a,1^{n-a-1}}.
\end{equation}
Furthermore, \pref{5.17} is an immediate consequence of the fact that the Compositional Shuffle Conjecture implies 
\begin{equation} \label{5.19}
\sum_{\beta\models n-a; } \nabla \, \BC_a \BC _\beta\un
= \sum_{\park\in\Parking_{n,= a}}
t^{\area(\park))} q^{\dinv(\park)} F_{\ides(\sigma(\park))},
\end{equation}
where the symbol $\Parking_{n,= a}$ signifies that the sum is to be extended over the parking functions whose Dyck path's first return to the diagonal occurs
exactly at row $a$.
\end{proof}

All these findings lead us to the following surprising identity.
\begin{thm} \label{thm5.4}
Let $\ret(\park)$ denote the first row where the supporting Dyck path of $\park$ hits the diagonal. We have, for all $n\geq 1$, that
\begin{equation} \label{5.20}
\Qop_{n,n} \unn = \sum_{\park\in\Parking_{n}} [\ret(\park)]_t\,  t^{\area(\park))-\ret(\park)+1} q^{\dinv(\park)} F_{\ides(\sigma(\park))}.
\end{equation}
\end{thm}

\begin{proof}[\bf Proof]
Combining \pref{5.7a} with \pref{5.10} gives
\begin{equation} \label{5.21}
\Qop_{n,n} \unn = (-qt)^{1-n} \sum_{k=0}^{n-1} \nabla (-qt)^k s_{n-k,1^k}.
\end{equation}
Now this may be rewritten as 
\begin{equation} \label{5.22}
\Qop_{n,n} \unn
= (-qt)^{1-n} \sum_{a=1}^{n} \nabla (-qt)^{n-a} s_{a,1^{n-a}}
= \sum_{a=1}^{n} \nabla (-qt)^{1-a} s_{a,1^{n-a}},
\end{equation}
and \pref{5.17} gives
\begin{eqnarray} 
\Qop_{n,n} \unn
&=& \sum_{a=1}^{n} t^{1-a} \sum_{\park \in\Parking_{n}} t^{\area(\park))} q^{\dinv(\park)} F_{\ides(\sigma(\park))}\, \charac\left(\ret(\park)\geq a \right) \label{5.23}\\ 
&=& \sum_{\park\in\Parking_{n}}
t^{\area(\park))} q^{\dinv(\park)} F_{\ides(\sigma(\park))} \sum_{a=1}^n t^{1-a}\, \charac \left( a\leq \ret(\park) \right).
\end{eqnarray}
This proves \pref{5.20}.
\end{proof}

\begin{rmk} \label{rmk5.3}
Extensive computer experiments have revealed that the following difference is Schur positive
\begin{equation} \label{5.24}
\Be_{km+1,kn} \cdot (-\mathbf{1}) - t^{d(km,kn)} \Be_{km,kn} \un, 
\end{equation}
where $d(km,kn)$ is the number of integral points between the diagonals for $(km+1,kn)$ and $(km,kn)$. Assuming that $m\leq n$ for simplicity sake, this implies that the following difference is also Schur positive
\begin{equation} \label{5.25}
\Be_{kn,kn} \un -  t^{a(km,kn)} \Be_{km,kn} \un,
\end{equation}
with $a(km,kn)$ equal to the area between the diagonal $(km,kn)$ and the diagonal $(kn,kn)$. This suggests that there is a nice interpretation of $t^{a(km,kn)}  \Be_{km,kn} \un$ as a new sub-module of the space of diagonal harmonic polynomials. We believe that we have a good candidate for this submodule, at least in the coprime case.
\end{rmk}

We terminate with some surprising observations concerning the so-called \define{Rational  $(q,t)$-Catalan}. In the present notation, this remarkable generalization of the $q,t$-Catalan polynomial (see \cite{qtCatPos}) may be defined by setting, for a coprime pair $(m,n)$ 
\begin{equation} \label{5.26}
C_{m,n}(q,t):= \left\langle \Qop_{m,n}\unn, e_n\right\rangle_*.
\end{equation}
It is shown in \cite{NegutShuffle}, by methods which are still beyond our reach, that this polynomial may also be obtained by the following identity.

\begin{thm} [A. Negut] \label{thm5.8}
\begin{equation} \label{5.27}
C_{m,n}(q,t)= \prod_{i=1}^n \frac{1}{(1-z_i)z_i^{a_i(m,n)}}
\prod_{i=1}^{n-1} \frac{1}{(1-qtz_i/z_{i+1})}
\prod_{1\leq i<j\leq n}\Omega[-M z_i/z_j]
\Big|_{z_1^0z_2^0\cdots z_n^0},
\end{equation}
with
\begin{equation} \label{5.28}
a_i(m,n):=\left\lfloor i\frac{m}{n}\right\rfloor- \left\lfloor (i-1)\frac{m}{n}\right\rfloor.
\end{equation}
\end{thm}

By a parallel but distinct path Negut has obtained the same polynomial as a weighted sum of standard tableaux. However his version of this result  turns out to be difficult to program on the computer. Fortunately, in \cite{constantTesler}, in a different but in a closely related context a similar sum over standard tableaux has been obtained. It turns out that basically the same method used in \cite{constantTesler} can be used in the present context to derive a standard tableaux expansion for $C_{m,n}(q,t)$ directly from \pref{5.27}. 
Let us write
%\begin{displaymath}
%\mathcal{N}_{m,n} [\mathbf{z};q,t]:= \prod_{i=1}^n \frac{1}{(1-z_i) z_i^{a_{n+1-i}(m,n)}}
%\prod_{i=2}^{n} \frac{1}{(1-qt z_i/z_{i-1})}
%\prod_{1\leq i <j \leq n} \Omega[-uMz_j/z_i],
%\end{displaymath}
\begin{eqnarray*}
\mathbf{z}_{m,n}&:=&\prod_{i=1}^n z_i^{a_{n+1-i}(m,n)}, \qquad {\rm and\ then}\\
\mathcal{N}_{m,n} [\mathbf{z};q,t]&:=& \frac{\Omega[\mathbf{z}]}{\mathbf{z}_{m,n}}
\, \prod_{i=2}^{n} \frac{1}{(1-qt z_i/z_{i-1})}
\prod_{1\leq i <j \leq n} \Omega[-uMz_j/z_i],
\end{eqnarray*}
where $\mathbf{z}$ stands for the set of variables $z_1,z_2,\ldots,z_n$. Equivalently, $\mathbf{z}=z_1+z_2+\ldots+z_n$ in the plethystic setup.
The resulting  rendition of the Negut's result can then be stated as follows.

\begin{prop} \label{thm5.9}
Let $T_n$ be the set of all standard tableaux with labels $1,2,\dots,n$. 
For a given $\tabl \in T_n$, we set $w_\tabl(k) = q^{j-1} t^{i-1}$ if the label $k$ of $\tabl$ is in the $i$-th row $j$-th column. 
We also denote by $S_\tabl$ the substitution set 
\begin{equation} \label{5.29}
\{z_k  = w_\tabl^{-1}(k): 1\leq k\leq n \}.
\end{equation}
We have
\begin{equation} \label{5.30}
C_{m,n}(q,t) = \sum_{\tabl\in T_n} \mathcal{N}_{m,n} [\mathbf{z};q,t] \prod_{i=1}^n (1-z_i w_\tabl(i)) \Big|_{S_\tabl},
\end{equation}
where the sum ranges over all standard tableaux of size $n$, and the $S_\tabl$ substitution should be carried out in the iterative manner. That is, we successively do the substitution for $z_1$ followed by cancellation, and then do the substitution
for $z_2$ followed by cancellation, and so on.
\end{prop}

\begin{proof}[\bf Proof]
For convenience, let us write $f(u)$ for $\Omega[-Mu]$ and $b_i$ for $a_{n+1-i}(m,n)$ this gives
$$
 \mathcal{N}_{m,n} [\mathbf{z};q,t] =
 \prod_{i=1}^n{1\over (1- z_i)z_i^{b_i}}\prod_{i=2}^n{1\over (1-qtz_i/z_{i-1})}\prod_{1\le i<j\le n}f(z_i/z_j).
$$ 
We will show by induction that   
\begin{displaymath}
\mathcal{N}_{m,n}[\mathbf{z};q,t] \Big|_{z_1^0\cdots z_{d}^0}  = \sum_{\tabl} \mathcal{N}_{m,n}[\mathbf{z};q,t] \prod_{i=1}^d (1-z_i w_\tabl(i)) \Big|_{S_\tabl},
\end{displaymath}
where $\tabl$ ranges over all standard tableaux of size $d$. Then the proposition is just the $d=n$ case. The $d=1$ case is straightforward. Thus, assume the equality holds for $d-1$.

Now for any term corresponding to a tableau $\tabl$ of size $d-1$, the factors containing $z_d$ are
\begin{eqnarray*}
 && \frac{1}{(1-z_d) z_d^{b_d}}\frac{1}{(1-qt\, {z_d/z_{d-1}})(1- qt \, {z_{d+1}/z_d})^{\chi(d<n)}} 
\prod_{1\leq i <d } f({z_d/z_i})\prod_{d <j \leq n} f({z_j}/{z_d})  \Big|_{S_\tabl} \\ 
&&=\ \frac{1}{(1-z_d) z_d^{b_d}}
\frac{1}{\left(1-qt\, z_d\,w_\tabl(d-1)\right)\left(1- qt \,{z_{d+1}}/{z_d} \right)^{\chi(d<n)} } 
\prod_{1\leq i <d } f(z_d\,w_\tabl(i))\prod_{d <j \leq n} f({z_j}/{z_d})\\ 
&&=\ \frac{1}
{z_d^{b_d} }\, \Omega\!\Big[z_d(1+qt\,w_\tabl(d-1) -M\sum_{i=1}^{d-1} w_\tabl(i)) \Big]
\prod_{d <j \leq n} f({z_j}/{z_d})\times\frac{1}{\big(1- qt \, {z_{d+1} \over{z_d}} \big)^{\chi(d<n)} }.
\end{eqnarray*}
By a simple calculation, first carried out in \cite{plethMac}  we may write
\begin{displaymath}
- M B_{\shape(\tabl)} = \bigg(\sum_{(i,j) \in \OC{\shape(\tabl)}} t^{i-1} q^{j-1}  - \sum_{(i,j) \in \IC{\shape(\tabl)}} t^i q^j \bigg) - 1
\end{displaymath}
where for a partition $\lambda$ we respectively denote by ``$\OC{\lambda}$'' and ``$\IC{\lambda}$'' the \define{outer} and \define{inner} corners of the Ferrers diagram of $\lambda$, as depicted in Figure~\ref{fig:corners} using the french convention. 
\begin{figure}[h]
\begin{center}\setlength{\unitlength}{3mm}
\includegraphics[width=5cm]{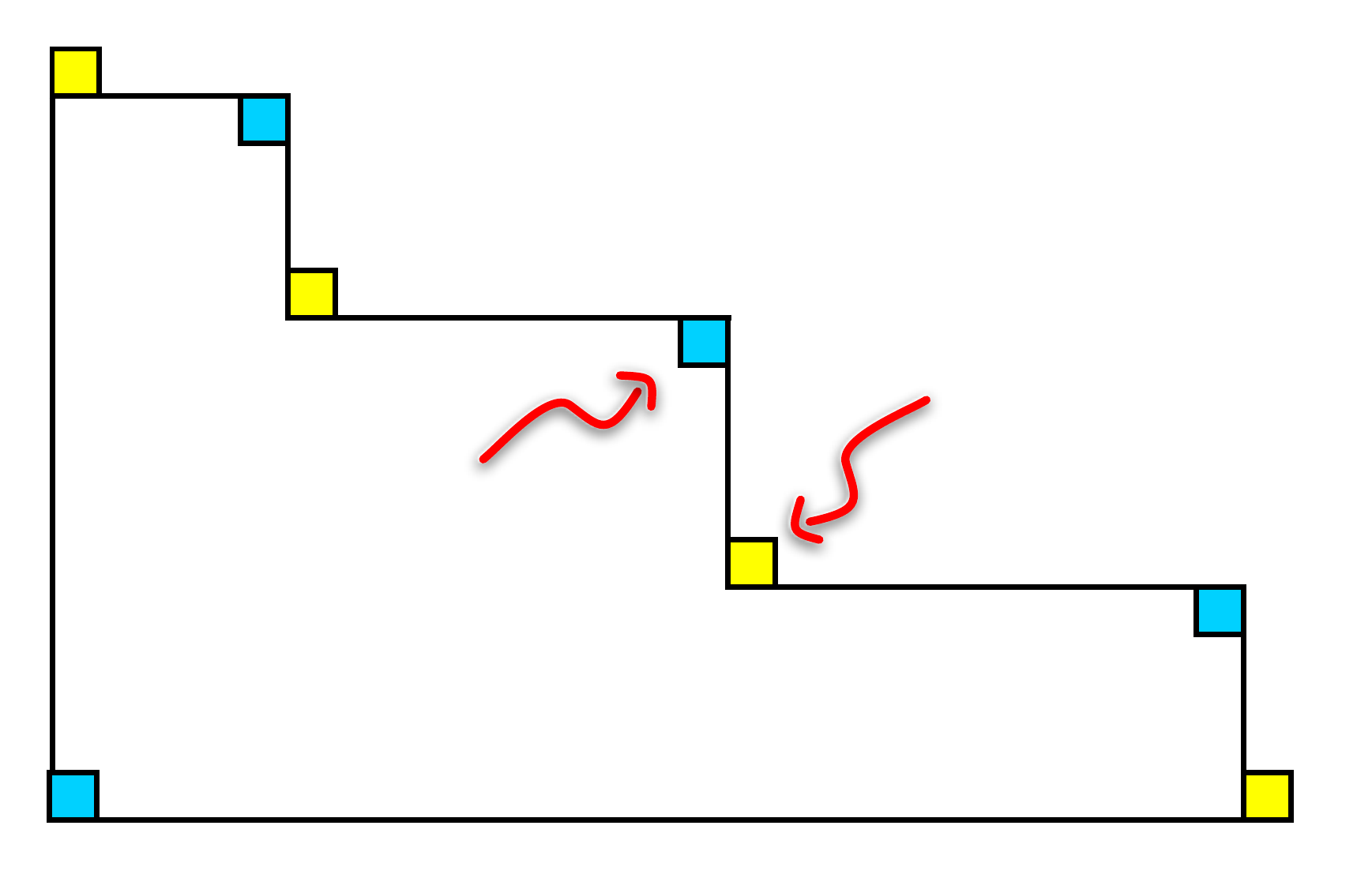}
     \begin{picture}(0,0)(0,0)
\put(-7,6.5){\rotatebox{-10}{\hbox{\tiny Outer corner}}}
\put(-14,4){\rotatebox{-10}{\hbox{\tiny Inner corner}}}
\end{picture}
\end{center}\vskip-20pt
\caption{Inner and outer corners of a partition.}
\label{fig:corners}
\end{figure}
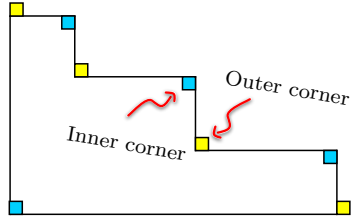

This given, the rational function object of the constant term becomes the following proper rational function in $z_d$ (provided $b_d \geq -1$):
$$
\frac{1}{z_d^{b_d}}
\frac{ \prod_{(i,j) \in \IC{\shape(\tabl)}} (1- t^i q^j  z_{d}) }
{ (1-{ z_d}\, qt\, w_\tabl(d-1))
\prod_{(i,j)\in  \OC{\shape(\tabl)}} (1- t^{i-1} q^{j-1} z_{d})}
\prod_{d <j \leq n}f({z_j}/{z_d})  
\times 
\frac{1}{\left(1- qt\, {z_{d+1} \over{z_d}} \right)^{\chi(d<n)}}.
$$
Since $d-1$ must appear in $\tabl$ in an inner corner of $\shape(\tabl)$, the factor $1-{ z_d}\, qt\, w_\tabl(d-1)$ in the denominator cancels with a factor in the numerator.

Therefore, the only factors, in the denominator that contribute to the constant term\footnote{By the partial fraction algorithm in \cite{fastAlg}.} are those of the form $(1-q^{j-1}t^{i-1} z_{d})$, for $(i,j)$ an outer corner of $\tabl$. For each such $(i,j)$, construct $\tabl'$ by adding $d$ to $\tabl$ at the cell $(i,j)$. 

Thus, by the partial fraction algorithm, we obtain
\begin{eqnarray*}
\mathcal{N}_{m,n}[\mathbf{z};q,t] \prod_{i=1}^{d-1} (1-z_i w_\tabl(i)) \Big|_{S_\tabl} \Big|_{z_d^0}
&=&\sum_{\tabl'} \mathcal{N}_{m,n}[\mathbf{z};q,t] \prod_{i=1}^d (1-z_i w_{\tabl'}(i)) \Big|_{S_\tabl} \Big|_{z_d={1\over w_{\tabl'}(d)} }\\
&=&\sum_{\tabl'} \mathcal{N}_{m,n}[\mathbf{z};q,t] \prod_{i=1}^d (1-z_i w_{\tabl'}(i)) \Big|_{S_{\tabl'}},
\end{eqnarray*}
where the sum ranges over all $\tabl'$ obtained from $\tabl$ by adding $d$ at one of its outer corners.

Applying the above formula to all $\tabl$ of size $d-1$, and using the induction hypothesis, we obtain:
\begin{eqnarray*}
\mathcal{N}_{m,n}[\mathbf{z};q,t] \Big|_{z_1^0\cdots z_{d}^0} 
&=& \sum_{\tabl} \mathcal{N}_{m,n}[\mathbf{z};q,t] \prod_{i=1}^{d-1} (1-z_i w_\tabl(i)) \Big|_{S_\tabl} \Big|_{z_d^0} \\ 
&=& \sum_\tabl \sum_{\tabl'} \mathcal{N}_{m,n}[\mathbf{z};q,t] \prod_{i=1}^d (1-z_i w_{\tabl'}(i)) \Big|_{S_{\tabl'}} \\
&=& \sum_{\tabl'} \mathcal{N}_{m,n}[\mathbf{z};q,t] \prod_{i=1}^d (1-z_i w_{\tabl'}(i)) \Big|_{S_{\tabl'}}, 
\end{eqnarray*}
where the final sum ranges over all $\tabl'$ of size $d$.
\end{proof}

\begin{rmk} \label{rmk5.4}
We can see that the above argument only needs $b_1\geq 0$, and $b_i\geq -1$ for $i=2,3,\dots,n$. Thus the equality of the right hand sides of \pref{5.27} and \pref{5.30} holds true also if the sequence $\{a_i(m,n)\}_{i=1}^n$ is replaced by any of these sequences. In fact computer data reveals that the constant term in \pref{5.27} yields a polynomial with positive integral coefficients for a variety of choices of  $(b_1,b_2,\dots,b_n)$ replacing the sequence $\{a_{i}(m,n)\}_{i=1}^n$. Trying to investigate the nature of these sequences and the possible combinatorial interpretations of the resulting polynomial led to the following construction.
\end{rmk}

Given a path $\path$ in the $m\times n$ lattice rectangle, we define the monomial of $\path$ by setting
\begin{equation} \label{5.31}
\mathbf{z}_\path:=\prod_{j=1}^n z_j^{e_j},
\end{equation}
where $e_j=e_j(\path)$ gives the number of east steps taken by $\path$ at height $j$. Note that, by the nature of \pref{5.31}, we are tacitly assuming that the path takes no east steps at height $0$. Note also that if $\path$ remains above the diagonal $(0,0)\to (m,n)$ then for each east step $(i-1,j)\to(i,j)$ we must have $i/j \leq m/n$. In particular for the path $\path_0$ that remains closest to the diagonal $(0,0)\to (m,n)$, the last east step at height $j$ must be given by $i=\lfloor j{m}/{n}\rfloor$. Thus
\begin{displaymath}
\mathbf{z}_{\path_0}=\mathbf{z}_{m,n}=\prod_{j=1}^n 
z_j^{ \lfloor j{m}/{n}\rfloor - \lfloor (j-1){m}/{n}\rfloor }.
\end{displaymath}
We can easily see that the series
\begin{displaymath}
\Omega[\mathbf{z}]= \prod_{j=1}^n\frac{1}{1- z_j}
\end{displaymath}
may be viewed as the generating function of all monomials of paths with north and east steps that end at height $n$ and start with a north step. We will refer to the later as the \define{NE-paths}. Our aim is to obtain a formula for the $q$-enumeration of the NE-paths in the $m\times n$ lattice rectangle that remain weakly above a given NE-path $\path$.

Notice that if 
\begin{equation} \label{5.32}
\mathbf{z}_\path = z_{r_1}z_{r_2}\cdots z_{r_m},
\qquad {\rm and} \qquad
\mathbf{z}_\delta=z_{s_1}z_{s_2}\cdots z_{s_m}.
\end{equation}
Then $\delta$ remains weakly above $\path$ if and only if
\begin{displaymath}
s_i \geq r_i \qquad\qquad \hbox{for }1\leq i\leq m.
\end{displaymath}
When this happens let us write $\delta \geq \path$. This given, let us set
\begin{displaymath}
C_\path(t):=\sum_{\delta \geq \path} t^{\area(\delta / \path)},
\end{displaymath}
where for $\path$ and $\delta$ as in \pref{5.32}, we let  $\area(\delta/\path)$ denote the number of lattice cells between $\delta$ and $\path$. In particular, for $\delta$ as in \pref{5.32}, we have
\begin{displaymath}
\area(\delta/\path)= \sum_{i=1}^m (s_i-r_i).
\end{displaymath}
Now we have the following  fact
\begin{prop} \label{prop5.5} For all path $\gamma$, we have
\begin{equation} \label{5.33}
\frac{\Omega[\mathbf{z}]}{\mathbf{z}_\path}\, 
\prod_{i=1}^{n-1} \frac{1}{1-tz_i/z_{i+1}}
\Big|_{z_1^0z_2^0\cdots z_n^0}
=\sum_{\delta \geq \path} t^{\area(\delta/\path)}.
\end{equation}
\end{prop}

\begin{proof}[\bf Proof]
Notice that each Laurent monomial produced by  expansion of the product
\begin{equation} \label{5.34}
\prod_{i=1}^{n-1} \frac{1}{1-tz_i/z_{i+1}}=
\prod_{i=1}^{n-1} \left(1+t \frac{z_i}{z_{i+1}}+\left(t\frac{z_i}{z_{i+1}} \right)^2+\cdots
\right)
\end{equation}
may be written in the form
\begin{eqnarray} \label{5.35}
\prod_{i=1}^{n-1}\left(t \frac{z_i}{z_{i+1}}\right)^{c_i}
&=&   \frac{z_1^{c_1}
 \prod_{i=2}^{n-1} z_i^{(c_i-c_{i-1})^+} }
{z_{n}^{c_{n-1}} \prod_{i=2}^{n-1} z_i^{(c_i-c_{i-1})^-}   }\ \prod_{i=1}^{n-1} t^{c_i} \nonumber\\[4pt]
&=& \frac{z_{a_1}z_{a_2}\cdots z_{a_\ell}}
{z_{b_1}z_{b_2}\cdots z_{b_\ell}}\  t^{\sum_{i=1}^{n-1}c_i}, \label{5.35}
\end{eqnarray}
with 
\begin{eqnarray*}
\ell &=&c_1+\sum_{i=2}^{n-1}(c_i-c_{i-1})^+ = \sum_{i=2}^{n-1}(c_i-c_{i-1})^-+c_{n-1},\\
a_r &=& \min \left\{ j : c_1+\sum_{i=2}^j(c_i-c_{i-1})^+ \geq r \right\},
\qquad {\rm and} \\
b_r &=& \min \left\{ j : \sum_{i=2}^j(c_i-c_{i-1})^- \geq r\right\}.
\end{eqnarray*}
Since $j=a_r$ forces $c_j>0$, we see that the equality
\begin{displaymath}
c_1+\sum_{i=2}^j(c_i-c_{i-1})^+=c_j+\sum_{i=2}^j(c_i-c_{i-1})^-
\end{displaymath}
yields that in \pref{5.35} we must have
\begin{equation} \label{5.36}
a_r< b_r , \qquad \hbox{for}\qquad 1\leq r\leq \ell.
\end{equation}
Now for the ratio in \pref{5.35} to contribute to the constant term in \pref{5.33}, it is necessary and sufficient that the reciprocal of this ratio should come out of the expansion
\begin{equation} \label{5.37}
\frac{\Omega[\mathbf{z}]}{\mathbf{z}_\path}
= \frac{1}{z_{r_1}z_{r_2}\cdots z_{r_m}}
\sum_{d_i\geq 0} z_1^{d_1}z_2^{d_2}\cdots z_n^{d_n}.
\end{equation}
That is for some $d_1,d_2,\dots,d_n$ we must have
\begin{equation} \label{5.38}
{ z_{r_1}z_{r_2}\cdots z_{r_m}\over  z_1^{d_1}z_2^{d_2}\cdots z_n^{d_n}}
= {z_{a_1}z_{a_2}\cdots z_{a_\ell}\over z_{b_1}z_{b_2}\cdots z_{b_\ell}}.
\end{equation}
Notice that $z_{a_1}z_{a_2}\cdots z_{a_\ell}$ and $z_{b_1}z_{b_2}\cdots z_{b_\ell}$ have no factor in common, since from the second expression in \pref{5.35} we derive that each variable $z_i$ can appear only in one of these two monomials. Thus $z_{a_1}z_{a_2}\cdots z_{a_\ell}$ divides $z_{r_1}z_{r_2}\cdots z_{r_m}$ and $z_{b_1}z_{b_2}\cdots z_{b_\ell}$ divides $z_1^{d_1}z_2^{d_2}\cdots z_n^{d_n}$ and in particular $\ell \leq m$. But this, together with the inequalities in \pref{5.36} shows that we must have
\begin{displaymath}
z_1^{d_1}z_2^{d_2}\cdots z_n^{d_n}
= z_{s_1}z_{s_2} \cdots  z_{s_m},
\qquad \hbox{with}\quad s_i\geq r_i \quad\hbox{for}\quad 1\leq i \leq m.
\end{displaymath}
In other words $z_1^{d_1}z_2^{d_2}\cdots z_n^{d_n}$ must  be the monomial of a NE-path $\delta\geq \path$. Moreover from the identity in \pref{5.38} we derive that
\begin{displaymath}
\area(\delta/\path) = (b_1-a_1)+(b_2-a_2)+\cdots +(b_\ell - a_\ell) 
= - \sum_{i=1}^\ell a_i + \sum_{i=1}^\ell b_i.
\end{displaymath}
Thus from the middle expression in 5.35 it follows that
\begin{eqnarray*}
\area(\delta/\path)
&=& -c_1-\sum_{i=2}^{n-1}i(c_i-c_{i-1})^+
+ \sum_{i=2}^{n-1}i(c_i-c_{i-1})^-+nc_{n-1} \\
&=& -c_1-\Big(\sum_{i=2}^{n-1}i(c_i-c_{i-1}) \Big)+c_{n-1}\\
&=& -c_1- \sum_{i=2}^{n-1}i c_i  +  \sum_{i=1}^{n-2}(i+1) c_i+nc_{n-1}\\
&=&\sum_{i=1}^{n-1}c_i,
\end{eqnarray*}
which is precisely the power of $t$ contributed by the Laurent monomial in \pref{5.35}.

Finally, suppose that $\delta$ is a NE-path weakly above $\path$ as in \pref{5.32}. This given, let us weight each lattice cell with southeast corner $(a,b)$ with the Laurent monomial $t z_b/z_{b+1}$. Then it is easily seen that for each fixed column $1\leq i\leq m$, the product of the weights of the lattice cells lying between $\delta$ and $\path$ is precisely $t^{s_i-r_i}x_{r_i}/x_{s_i}$. Thus
\begin{displaymath}
\prod_{i=1}^m t^{s_i-r_i} \frac{x_{r_i}}{x_{s_i}}
= t^{\area(\delta/\path)} \frac{\mathbf{z}_\path}{\mathbf{z}_\delta}.
\end{displaymath}
Since the left hand side of this identity is in the form given in \pref{5.35}, we clearly see that every summand of $C_\path(t)$ will come out of the constant term.
\end{proof}

\begin{rmk} \label{rmk5.5}
It is easy to see that, for $q=1$, the constant term in \pref{5.27}, reduces to the one in \pref{5.33} with $\path=\path_0$ (the closest path to the diagonal $(0,0)\to (m,n)$). This is simply due to the identity
\begin{displaymath}
\Omega[-uM]\Big|_{q=1}= {(1-u)(1-qtu)\over (1-tu)(1-qu)}\Big|_{q=1}= 1.
\end{displaymath}
Moreover, since the coprimality of the pair $(m,n)$ had no role in the proof of Proposition \ref{prop5.5}, we were led to the formulation of the following conjecture, widely supported by our computer data.
\end{rmk}

\begin{conj} \label{conjV}
For any pair of positive integers $(u,v)$ and any NE-path $\path$ in the $u\times v$ lattice that remains weakly above the lattice diagonal $(0,0)\to (u,v)$ we have
\begin{equation} \label{5.39}
\frac{\Omega[\mathbf{z}]}{\mathbf{z}_\path}\,
\prod_{i=1}^{n-1} \frac{1}{(1-qtz_i/z_{i+1})}
\prod_{1\leq i<j\leq n} \Omega[-M z_i/z_j]
\Big|_{z_1^0z_2^0\cdots z_n^0}
= \sum_{\delta \geq \path} t^{\area((\delta/\path)} q^{\dinv(\delta)},
\end{equation}
where $\dinv(\delta)$ is computed as in step (3) of Algorithm \ref{algL} for $(km,kn)=(u,v)$.
\end{conj}

Remarkably, the equality in \pref{5.39} is still unproven even for general coprime pairs $(m,n)$, except of course for the cases $m=n+1$ proved in \cite{qtCatPos}. What is really intriguing is to explain how the inclusion of the expression
\begin{displaymath}
\prod_{1\leq i<j\leq n}\Omega[-M z_i/z_j]
\end{displaymath}
accounts for the insertion of the factor
\begin{displaymath}
q^{\dinv(\path)-\area(\delta/\path)}
\end{displaymath}
in the right hand side of \pref{5.33}). A combinatorial explanation of this phenomenon would lead to an avalanche of consequences in this area, in addition to proving Conjecture \ref{conjV}.

%%%%%%%%%%%%%%%%%%%%%%%%%%%%%%%%%%

%Temporary ending. More plain tex below

%\bibliographystyle{abbrvnat}
%\bibliography{Compo}
%\label{sec:biblio}

%\begin{thebibliography}{22}
%\providecommand{\natexlab}[1]{#1}
%\providecommand{\url}[1]{\texttt{#1}}
%\expandafter\ifx\csname urlstyle\endcsname\relax
%  \providecommand{\doi}[1]{doi: #1}\else
%  \providecommand{\doi}{doi: \begingroup \urlstyle{rm}\Url}\fi

\end{document}